\documentclass[10pt,reqno]{amsart}

\usepackage{enumerate}
\usepackage{amsmath, amssymb, amsthm,amsfonts}
\usepackage{mathrsfs}
\usepackage{esint}
\usepackage{xcolor}
\usepackage{mathtools}
\usepackage{hyperref}
\usepackage{bm}
\usepackage{accents}
\usepackage[foot]{amsaddr}
\usepackage{todonotes}

\makeatletter

\numberwithin{equation}{section}

\newtheorem{thm}{Theorem}[section]
\newtheorem{theorem}[thm]{Theorem}
\newtheorem{lemma}[thm]{Lemma}

\newtheorem{proposition}[thm]{Proposition}
\theoremstyle{definition}

\newtheorem{definition}[thm]{Definition}
\newtheorem{example}[thm]{Example}
\newtheorem{assumption}[thm]{Assumption}
\theoremstyle{remark}

\newtheorem{remark}[thm]{\bf{Remark}}

\newcommand\bD{\mathbb{D}}
\newcommand\bE{\mathbb{E}}

\newcommand\bH{\mathbb{H}}

\newcommand\bL{\mathbb{L}}

\newcommand\bN{\mathbb{N}}

\newcommand\bP{\mathbb{P}}

\newcommand\bR{\mathbb{R}}
\newcommand\bS{\mathbb{S}}

\newcommand\bZ{\mathbb{Z}}

\newcommand\fB{\mathbf{B}}

\newcommand\fH{\mathbf{H}}

\newcommand\cD{\mathcal{D}}
\newcommand\cE{\mathcal{E}}
\newcommand\cF{\mathcal{F}}

\newcommand\cL{\mathcal{L}}
\newcommand\cM{\mathcal{M}}
\newcommand\cN{\mathcal{N}}

\newcommand\cP{\mathcal{P}}

\newcommand\cR{\mathcal{R}}
\newcommand\cS{\mathcal{S}}
\newcommand\cT{\mathcal{T}}

\newcommand\frH{\mathfrak{H}}

\begin{document}

\title[]{The Neumann problem for stochastic partial differential equations}

\author{Kyeong-Hun Kim$^{1}$}
\address{$^1$ Department of Mathematics, Korea University, 145 Anam-ro, Seongbuk-gu, Seoul,
02841, Republic of Korea}
\email{kyeonghun@korea.ac.kr}

\author{Junhee Ryu$^{2}$}
\address{$^2$ School of Mathematics, Korea Institute for Advanced Study, 85 Hoegi-ro, Dongdaemun-gu, Seoul, 02455, Republic of Korea} \email{junhryu@kias.re.kr}
\thanks{K.-H. Kim was supported by the National Research Foundation of Korea(NRF) grant funded by the Korea government(MSIT) (No. RS-2025-00556160). J. Ryu was supported by a KIAS Individual Grant (MG101501) at Korea Institute for Advanced Study and by the National Research Foundation of Korea (NRF) grant funded by the Korea government (MSIT) (No. RS-2026-25477288).}



\begin{abstract}
In this paper, we study the Neumann problem for second-order stochastic partial differential equations in the half space. The coefficients are random and merely measurable in time. We develop a class of weighted function spaces for this problem, and establish existence, uniqueness, and regularity results. In contrast to the Dirichlet problem, the Neumann problem admits solutions with improved boundary regularity under suitable assumptions on the coefficients and data.
\end{abstract}

\maketitle

\section{Introduction}

Stochastic partial differential equations (SPDEs) arise as models of phenomena involving random effects, and their mathematical theory has developed over the past several decades.
In the case of equations defined on domains, boundary conditions play an important role in both solvability and the regularity of solutions.
In this paper, we study second-order SPDEs with homogeneous Neumann boundary conditions in the half space. More precisely, we consider
\begin{equation} \label{maineq}
\begin{cases}
du = \left(a^{ij}(t)D_{ij}u+f(u)\right)dt + \left(\sigma^{ik}(t)D_iu+g^k(u)\right)dw_t^k,\ t>0, x\in\bR^d_+
,\\
D_du|_{x_d=0}=0,
\\
u(0,x)=u_0.
\end{cases}
\end{equation}
Here, $\bR^d_+:=\{x\in\bR^d:x_d>0\}$ is the upper half space, $\{w_t^k\}_{k\in\bN}$ is a sequence of independent one-dimensional Wiener processes, and $D_du$ denotes the derivative of $u$ in the normal direction. The leading coefficients $a^{ij}$ and $\sigma^{ik}$ are predictable in $(t,\omega)$, possibly random and merely measurable in time, and satisfy the stochastic parabolicity condition.
The indices $i$ and $j$ go from $1$ to $d$ and $k$ runs through $\{1,2,\dots\}$. Einstein's summation convention on $i,j$, and $k$ is assumed throughout the paper.
The terms $f(u)$ and $g(u)$ are semi-linear in $u$. These conditions are stated precisely in Assumptions \ref{as_coeff} and \ref{as_semi}.

Our goal is to develop a weighted Sobolev framework for the Neumann problem that covers data of arbitrary real regularity, and to establish existence, uniqueness, and regularity results in this framework. 
Combined with the Sobolev-H\"older embedding, our estimates yield improved boundary regularity compared with the corresponding Dirichlet theory. For any $\varepsilon\in(0,1)$ and sufficiently regular data, the solution is spatially $C^{2-\varepsilon}$ up to the boundary when $\sigma^{ik}\equiv0$, and $C^{1-\varepsilon}$ up to the boundary in the presence of the gradient term $\sigma^{ik}D_iu$ in the stochastic part.

This is in contrast to the Dirichlet problem. In \cite{K94}, it was shown that even for smooth deterministic $g$, the second spatial derivatives of the solution to the homogeneous Dirichlet problem may not be square integrable.
To overcome this difficulty, the problem was formulated in weighted Sobolev spaces, where derivatives of the solution are allowed to become singular near the boundary. We refer the reader to \cite{K04,Kry_halfspace,KLline}.
In this setting, a weighted Sobolev-H\"older embedding theorem provides at best spatially $C^{1-\varepsilon}$ regularity up to the boundary.
In the presence of stochastic gradient terms, the situation for the Dirichlet problem is even worse. In one dimension, the author of \cite{K03} showed that the boundary H\"older exponent can be arbitrarily small even for a linear equation with constant coefficients and smooth initial data. See Remark \ref{rem1013} for more details.
A similar phenomenon is expected in higher dimensions, although the optimal regularity remains open. Nevertheless, in \cite{M19}, some positive H\"older regularity was obtained.

There have been relatively few studies of SPDEs with Neumann boundary conditions. In \cite{SV03}, the authors considered non-autonomous semi-linear equations driven by spatially colored noise, with conormal boundary conditions on smooth bounded convex domains. They established existence and uniqueness of solutions, together with $L_2$-based Sobolev regularity.
We also refer the reader to \cite[Example 8.2]{V10}, where stochastic maximal $L_2$-regularity was derived for related conormal problems.
In \cite{NVW12}, stochastic maximal $L_p$-regularity was established using the bounded $H^\infty$-calculus. This theory yields $W^{1,q}$ estimates for stochastic convolutions generated by the Neumann Laplacian.
Corresponding mixed-norm estimates with temporal weights were obtained for time-dependent conormal problems in \cite[Example 8.17]{LV21}.

The theory for Neumann problems involving the gradient term $\sigma^{ik}D_iu$ in the stochastic part is less developed. In \cite{GS24}, the authors treated fully nonlinear Neumann SPDEs in a viscosity framework with spatially homogeneous noise Hamiltonians $H^k(Du)$. Their analysis does not cover general spatially dependent semi-linear noise terms of the form $g^k(t,x,u)$.
See also \cite{CT16} for singular diffusion equations with gradient Stratonovich noise under geometric assumptions.

In this paper, we deal with \eqref{maineq} in an $L_p$-Sobolev framework in which the data $f$, $g$, and $u_0$ may have arbitrary real regularity and the noise coefficients $\sigma^{ik}$ are allowed to be nonzero.
Each of these two points presents a different difficulty.
The first difficulty is related to the Neumann boundary condition. If $u\in W_p^s(\bR^d_+)$, then the trace of $D_du$ is well defined only when $s>1+1/p$.
For deterministic equations, the Neumann problem has been studied with the boundary condition understood in the trace or weak sense. See \cite{DHP07,DL21,DL22,FMM98,H21,DKZ16,KK22,KN14,LLRV25}.
In our setting, however, the trace of $D_du$ may not exist, since we aim to cover arbitrary real regularity. We instead use a weighted condition on $D_du$, which gives a formulation of the Neumann condition without a boundary trace.
The second difficulty is related to the stochastic gradient term. On the half space, this term interacts with the boundary and leads to a more restrictive weighted theory than in the case $\sigma^{ik}\equiv0$.

To address the first difficulty, we use the weighted Sobolev spaces $H^\gamma_{p,\theta}(\bR^d_+)$ introduced in \cite{Khalf} as base spaces. Here, $p\in(1,\infty)$ is the integrability exponent, $\gamma\in\bR$ is the regularity index, and $\theta\in\bR$ determines the weight. When $\gamma$ is a nonnegative integer, an equivalent norm is given by
\begin{equation*}
\|u\|_{H_{p,\theta}^\gamma} = \sum_{m=0}^{\gamma}  \left(\int_{\bR^d_+} |x_d^mD_x^mu(x)|^p x_d^{\theta-d} dx\right)^{1/p}.
\end{equation*}
The weights $x_d^m$ allow derivatives of higher order to be more singular near the boundary, while $x_d^{\theta-d}$ controls their integrability near the boundary. This space can be generalized in a unified manner for general $\gamma\in\bR$ (see Section \ref{sec_weighted}). In particular, $H_{p,\theta}^2(\bR^d_+)$ controls $u$, $x_dD_xu$, and $x_d^2D_x^2u$ in $L_{p,\theta}(\bR^d_+) = L_p(\bR^d_+,x_d^{\theta-d}dx)$.
In contrast, the classical weighted second-order formulation controls $u$, $D_xu$, and $D_x^2u$ in this same base space.
Accordingly, we place $u$, $D_xu$, and $D_x^2u$ in the same space $H^\gamma_{p,\theta}$ and consider solutions satisfying
\begin{equation*}
u,D_xu,D_x^2u \in \bH^\gamma_{p,\theta}(\tau),\qquad x_d^{-1}D_du \in \bH^{\gamma+1}_{p,\theta}(\tau).
\end{equation*}
Here, $\bH_{p,\theta}^{\gamma}(\tau) := L_p((0,\tau];H_{p,\theta}^{\gamma}(\bR^d_+))$, where $\tau$ is a bounded stopping time and $(0,\tau]:=\{(\omega,t): 0<t\leq\tau(\omega)\}$.
When $\gamma=0$ and $\theta=d$, the first three conditions correspond to the standard $W_p^2$ regularity.
The regularity in these spaces can be understood as follows. Away from the boundary, we obtain the usual local Sobolev regularity of order $\gamma+2$. Near the boundary, higher derivatives are instead controlled in the weighted structure of $H^\gamma_{p,\theta}$.

The condition $x_d^{-1}D_du(t,\cdot) \in H^{\gamma+1}_{p,\theta}$ reflects the homogeneous Neumann boundary condition. For instance, if $u\in C_c^\infty(\overline{\bR^d_+})$ and $d-1<\theta<d-1+p$, then the condition $x_d^{-1}D_du \in L_{p,\theta}(\bR^d_+)$ is equivalent to $D_du(x',0)=0$. 
This observation leads us to impose a weighted condition on the
normal derivative for every $\gamma\in\bR$, without imposing any condition on the boundary
value of the solution. When the regularity and weight parameters are chosen as in Remark \ref{rem5250111} $(ii)$, the normal derivative extends continuously to the boundary and vanishes there, so the Neumann condition holds in the classical sense.

For the second difficulty, recall that on the whole space the stochastic gradient term can be removed by a stochastic change of variables. On the half space, however, this argument is not suitable when $\sigma^{ik}\neq0$, since the change of variables transforms the half space into a moving random domain. We instead observe that, for the linear equation, $D_du$ satisfies a Dirichlet SPDE with differentiated data. The weighted Dirichlet theory in \cite{Kry_halfspace} then provides estimates for $x_d^{-1}D_du$, while tangential derivatives are estimated by whole-space SPDE estimates in the tangential variables.

In this way, we overcome both difficulties and establish existence, uniqueness, and an a priori estimate for every $\gamma\in\bR$ and $p\in[2,\infty)$ on bounded random time intervals. For the precise statement, we refer the reader to Theorem \ref{thm_main}. Moreover, boundary and interior H\"older estimates are stated in Theorems \ref{thm_bdry} and \ref{thm_int}.

The paper is organized as follows.
In Section \ref{sec_main}, we introduce the solution spaces and state the main results.
In Section \ref{sec_funct}, we establish the interpolation and lifting properties needed to treat arbitrary real regularity and compare our spaces with weighted Triebel-Lizorkin spaces.
In Section \ref{sec_simple}, we construct solutions for equations with the Laplacian using the Neumann heat kernel and establish the corresponding estimates.
In Section \ref{sec_apriori}, we prove the a priori estimates for general coefficients.
Finally, in Section \ref{sec_proof}, we complete the proofs of the main results. The solvability for general coefficients follows from the method of continuity, the semi-linear terms are handled by an iteration argument, and the H\"older regularity up to the boundary is obtained from the comparison with weighted Triebel-Lizorkin spaces.

We finish this section by introducing some notation used throughout the paper. We use ``$:=$'' or ``$=:$'' to denote a definition. By $\bN$ and $\bZ$, we denote the natural numbers and integers, respectively. We write $\bN_0:=\bN\cup\{0\}$. As usual, $\bR^d$ stands for the Euclidean space of points $x=(x_1,\dots,x_d)=(x',x_d)$. We also write $\bR:=\bR^1$ and $\bR_+:=\bR^1_+=(0,\infty)$.
We use $D_x^nu$ to denote the partial derivatives of order $n\in\bN_0$ with respect to the space variables, and $D_xu:=D_x^1u$. We also write
\begin{equation*}
D_iu=\frac{\partial u}{\partial x_i},\qquad D_{ij}u=\frac{\partial^2u}{\partial x_i\partial x_j}.
\end{equation*}
For an open set $D\subset\bR^d$, by $C_c^\infty(D)$ we denote the space of infinitely differentiable functions with compact support in $D$. We also write $C_c^\infty(\overline{D})$ for the set of all infinitely differentiable functions on $D$ whose derivatives of all orders extend continuously to $\overline{D}$ and whose support is compact in $\overline{D}$.
By $\cD'(\bR^d_+)$, we denote the space of all distributions on $\bR^d_+$. For $u\in\cD'(\bR^d_+)$ and $\phi\in C_c^\infty(\bR^d_+)$, we write
\begin{equation*}
(u,\phi):=(u,\phi)_{\bR^d_+}:=u(\phi).
\end{equation*}

\section{Preliminaries and main results} \label{sec_main}

In this section, we introduce the function spaces used to study \eqref{maineq}. More specifically, in Section \ref{sec_weighted}, we present the basic settings and the solution space. Then the main results are stated in Section \ref{sec_res}.

Throughout this paper, we fix a complete probability space $(\Omega,\cF,\bP)$ equipped with a filtration $\{\cF_t, t\geq0\}$ satisfying the usual conditions.
Let $\cP$ denote the predictable $\sigma$-field generated by $\cF_t$, that is, $\cP$ is the smallest $\sigma$-field containing all sets of the form $A\times (s,t]$ where $s<t$ and $A\in \cF_s$. Let $\{w^k_t\}_{k\in\bN}$ be a sequence of  independent one-dimensional Wiener processes adapted to the filtration $\{\cF_t, t\geq0\}$.

\subsection{Weighted Sobolev spaces} \label{sec_weighted}

We first present the definitions of Sobolev and Besov spaces on $\bR^d$. For $p\in(1,\infty)$ and $\gamma\in\bR$, the Sobolev space $H_p^\gamma=H_p^\gamma(\bR^d)$ is defined as the set of all tempered distributions $u\in \cS'(\bR^d)$ such that
\begin{equation*}
    \|u\|_{H_p^\gamma}=\|(1-\Delta)^{\gamma/2}u\|_{L_p}= \left\| \cF^{-1} \left\{ (1+|\cdot|^2)^{\gamma/2}\cF(u)(\cdot) \right\} \right\|_{L_p} <\infty.
\end{equation*}
Here, $\cF$ and $\cF^{-1}$ denote the standard Fourier transform and its inverse on $\bR^d$, respectively. In particular, when $\gamma\in\bN$,
$$
H_p^\gamma=W_p^\gamma:=\{u: D_x^\beta u \in L_p, |\beta|\leq \gamma\}.
$$
We also consider the $l_2$-valued Sobolev space $H_p^\gamma(l_2)=H_p^\gamma(\bR^d;l_2)$, consisting of all $l_2$-valued distributions $v$ on $\bR^d$ satisfying
$$
\|v\|_{H_p^\gamma(l_2)}:=\| |(1-\Delta)^{\gamma/2}v|_{l_2} \|_{L_p}<\infty,
$$
where $l_2$ is the set of all sequences $a=(a^1,a^2,\cdots)$ such that
$$
|a|_{l_2}:=\left(\sum_{k=1}^\infty |a^k|^2 \right)^{1/2}<\infty.
$$
Now we introduce the Besov space as a suitable space for the initial data. Let us choose a function $\Psi$ whose Fourier transform $\cF(\Psi)$ is infinitely differentiable, supported in an annulus $\{\xi\in\bR^d : \frac{1}{2} \leq |\xi| \leq 2\}$, $\cF(\Psi)\geq0$, and
$$
\sum_{j\in\bZ} \cF(\Psi)(2^{-j}\xi)=1, \quad \forall \xi\neq0.
$$
For $u\in\cS'(\bR^d)$ and $j\in \bZ$, define
\begin{equation} \label{eq4070111}
    \Delta_j u(x):= \cF^{-1}\left( \cF(\Psi)(2^{-j}\cdot) \cF(u) \right)(x), \quad S_0 u(x) :=  \cF^{-1}\left( \chi \cF(u)\right)(x),
\end{equation}
where $\chi(\xi) := 1-\sum_{j=1}^\infty \cF(\Psi)(2^{-j}\xi)$.
Then for $p>1$ and $\gamma\in\bR$, the Besov space $B_{p,p}^\gamma=B_{p,p}^\gamma(\bR^d)$ is the set of all $u\in \cS'(\bR^d)$ such that
\begin{equation*}
    \|u\|_{B_{p,p}^\gamma}:=\|S_0u\|_{L_p} + \left( \sum_{j=1}^\infty 2^{j p \gamma} \|\Delta_j u\|_{L_p}^p \right)^{1/p} <\infty.
\end{equation*}

From now on, we consider function spaces defined on domains. We recall the definition of $H_{p,\theta}^\gamma$, which was introduced in \cite{Khalf} to deal with homogeneous Dirichlet boundary conditions.
For $p\in(1,\infty)$, $\theta,\gamma\in\bR$, and $\zeta(x_d)\in C_c^\infty(\bR_+)$ such that $\zeta\geq0$ and
\begin{equation} \label{eq5062057}
    \sum_{m\in\bZ} \zeta^p(e^{x-m})\geq c>0, \quad x\in \bR,
\end{equation}
we denote by $H_{p,\theta}^\gamma$ the space of distributions $u\in \cD'(\bR^d_+)$ such that
\begin{align} \label{defH}
\|u\|^p_{H_{p,\theta}^\gamma}:= \sum_{m\in\bZ} e^{m\theta} \| u(e^m\cdot) \zeta\|^p_{H_p^\gamma}<\infty.
\end{align}
In particular, when $\gamma\in\bN_0$,
\begin{equation*}
H_{p,\theta}^{\gamma} = \{ u\in \cD'(\bR^d_+): u, x_dD_xu,\cdots, x_d^\gamma D_x^{\gamma}u \in  L_{p,\theta}\},
\end{equation*}
where $L_{p,\theta}:=H_{p,\theta}^0=L_p(\bR^d_+,x_d^{\theta-d}dx)$ (see \cite[Corollary 2.3]{Khalf}). Thus, \eqref{defH} is a unified definition of these spaces for arbitrary $\gamma\in\bR$.
The function space can be extended to the $l_2$-valued function space $H_{p,\theta}^\gamma(l_2)$, which is defined as
\begin{equation*}
    \|u\|^p_{H_{p,\theta}^\gamma(l_2)}:= \sum_{m\in\bZ} e^{m\theta} \| u(e^m\cdot)\zeta\|^p_{H_p^\gamma(l_2)}<\infty.
\end{equation*}
 We remark that both $H_{p,\theta}^\gamma$ and $H_{p,\theta}^\gamma(l_2)$ are independent of the choice of $\zeta$ (see \cite[Lemma 1.4]{Khalf}). Similar to $L_{p,\theta}:=H_{p,\theta}^0$, when $\gamma=0$, we denote  $L_{p,\theta}(l_2):= H_{p,\theta}^{0}(l_2)$.
Some other properties of $H_{p,\theta}^{\gamma}$ are presented in Section \ref{sec_funct}.

Let us consider the deterministic equation
\begin{equation*}
\partial_t u = \Delta u + f, \quad t>0, x\in \bR^d_+,
\end{equation*}
with the Neumann condition $D_du(x',0)=0$. The classical theory is usually formulated in the setting where the solution $u$ belongs to
\begin{equation*}
W_p^2(\bR^d_+):= \{u\in L_p(\bR^d_+): u,D_xu, D_x^2u \in L_p(\bR^d_+) \},
\end{equation*}
and $f\in L_p(\bR^d_+)$. This framework can be extended to the weighted setting 
$$
W_p^2(\bR^d_+,x_d^{\theta-d}):=\{u\in L_{p,\theta}: u,D_xu,D_x^2u \in L_{p,\theta}\}.
$$
Our goal is to study stochastic equation \eqref{maineq} within this framework. As a further extension, in order to handle data of arbitrary regularity, we introduce
\begin{equation*}
    H_{p,\theta}^{\gamma,(n)} := \{u: u, D_xu, \cdots, D_x^n u\in H_{p,\theta}^{\gamma-n}\}, \quad n\in \bN_0,
\end{equation*}
equipped with the norm
\begin{equation*}
    \|u\|_{H_{p,\theta}^{\gamma,(n)}} := \sum_{k=0}^{n} \|D_x^ku\|_{H_{p,\theta}^{\gamma-n}}.
\end{equation*}
Similarly, we define the $l_2$-valued function space $H_{p,\theta}^{\gamma,(n)}(l_2)$.
We remark that $H_{p,\theta}^{n,(n)}$ coincides with the classical weighted Sobolev space
\begin{equation*}
    W_{p}^{n}(\bR^d_+,x_d^{\theta-d}):=\{u\in L_{p,\theta}: u,\cdots,D_x^nu \in L_{p,\theta} \}.
\end{equation*}

Next, we introduce stochastic Banach spaces.
Let us denote 
\begin{eqnarray*}
    \bH_{p,\theta}^{\gamma,(n)}(\tau):= L_p((0,\tau],\cP;H_{p,\theta}^{\gamma,(n)})
\end{eqnarray*}
and
\begin{equation*}
    \fH_{p,\theta}^{\gamma,(n)} = L_{p}(\Omega,\cF_0; H_{p,\theta}^{\gamma,(n)}).
\end{equation*}
In the case when $n=0$, we simply write $\bH_{p,\theta}^{\gamma}(\tau):=\bH_{p,\theta}^{\gamma,(0)}(\tau)$ and $\fH_{p,\theta}^{\gamma} := \fH_{p,\theta}^{\gamma,(0)}$. When $\gamma=n=0$, we also write $\bL_{p,\theta}(\tau):=\bH_{p,\theta}^0(\tau)$.
The $l_2$-valued function space $\bH_{p,\theta}^{\gamma,(n)}(\tau,l_2)$ is defined in the usual way.

For a sufficiently nice function $u$, say $u\in C_c^\infty(\overline{\bR^d_+})$, the condition $u(x',0)=0$ is equivalent to $x_d^{-1}u \in L_{p,\theta}$, provided that $\theta\in(d-1,d-1+p)$ (see e.g. \cite{K08nonzero,KKW24}).
Motivated by this observation, we introduce a space that provides a unified framework for the Neumann problem for all regularity indices $\gamma\in\bR$. For $\gamma\in\bR$, $p\in(1,\infty)$, and $\theta\in\bR$, we denote
\begin{equation*}
    \widetilde{H}_{p,\theta}^{\gamma+2,(2)} := \{u\in H_{p,\theta}^{\gamma+2,(2)}: x_d^{-1}D_du\in H_{p,\theta}^{\gamma+1}\},
\end{equation*}
where the norm on this space is defined by
\begin{equation*}
    \|u\|_{\widetilde{H}_{p,\theta}^{\gamma+2,(2)}} := \|u\|_{H_{p,\theta}^{\gamma+2,(2)}} + \|x_d^{-1}D_du\|_{H_{p,\theta}^{\gamma+1}}.
\end{equation*}
We also define
\begin{equation*}
    \widetilde{\fH}_{p,\theta}^{\gamma+2,(2)} = L_{p}(\Omega,\cF_0; \widetilde{H}_{p,\theta}^{\gamma+2,(2)}), \quad \widetilde{\bH}_{p,\theta}^{\gamma+2,(2)}(\tau) := L_p((0,\tau],\cP;\widetilde{H}_{p,\theta}^{\gamma+2,(2)}).
\end{equation*}

We now discuss the space for the initial data. Let $T>0$ be nonrandom, and $X_0\subset X_1$ be Banach spaces. Suppose that $u\in L_p((0,T);X_0)$ and $\partial_t u \in L_p((0,T);X_1)$. Then it is known that the space for initial data is characterized by real interpolation:
\begin{equation*}
    u(0,\cdot) \in (X_0,X_1)_{1/p,p},
\end{equation*}
where $(A,B)_{\vartheta,q}$ is the real interpolation space of $A$ and $B$ with parameters $\vartheta\in(0,1)$ and $q\in(0,\infty)$ (see e.g. \cite[Definition 2.4]{Triebel1} for the precise definition, and \cite{W95, LV20, CLSW26} for this characterization).
Thus, we consider
\begin{equation*}
   U_{p,\theta}^{\gamma+2} = (\widetilde{\fH}_{p,\theta}^{\gamma+2,(2)},\fH_{p,\theta}^{\gamma})_{1/p,p}.
\end{equation*}
 In particular, by \cite[Theorem 2.2.10]{HNVW16},
 \begin{equation*}
     U_{p,\theta}^{\gamma+2} = L_{p}\left(\Omega,\cF_0; (\widetilde{H}_{p,\theta}^{\gamma+2,(2)},H_{p,\theta}^{\gamma})_{1/p,p}\right).
 \end{equation*}

We next introduce the distributional formulation of the stochastic equation.

\begin{definition}
    Let $\tau$ be a stopping time, $D\subset \bR^d$ be an open set, $u_0$ be a $\cD'(D)$-valued $\cF_0$-measurable function defined on $\Omega$, and $u,f$, and $g^k$ be $\cD'(D)$-valued predictable functions defined on $(0,\tau]$.
    Suppose that for any $\phi\in C_c^\infty(D)$,
    \begin{equation*}
        (f,\phi) \in L_1((0,\tau]), \quad |(g^k,\phi)|_{l_2} \in L_2((0,\tau]).
    \end{equation*}
    Then we say that $u$ satisfies
\begin{equation} \label{sense}
    du = fdt + g^kdw_t^k, \quad t\in (0,\tau], x\in D; \quad u(0,\cdot)=u_0
\end{equation}
in the sense of distributions if for any $\phi\in C_c^\infty(D)$,
\begin{equation*}
    (u(t,\cdot),\phi) = (u(0,\cdot),\phi) +  \int_0^t (f(s,\cdot),\phi) ds + \int_0^t (g^k(s,\cdot),\phi) dw_s^k
\end{equation*}
holds for all $t\leq \tau$ (a.s.).
\end{definition}

Now we introduce the solution space for the SPDEs studied in this paper.

\begin{definition}
    We say that $u\in \frH_{neu}^{\gamma+2}(p,\theta,\tau)$ if $u \in \widetilde{\bH}_{p,\theta}^{\gamma+2,(2)}(\tau)$ and there are $u_0\in U_{p,\theta}^{\gamma+2}$, $f\in \bH_{p,\theta}^\gamma(\tau)$, and $g\in {\bH}_{p,\theta}^{\gamma+1,(1)}(\tau,l_2)$ such that $u$ satisfies \eqref{sense} in the sense of distributions. In this case, we write $\bD u:= f$ and $\bS u:=g$. Moreover,
\begin{align} \label{def_norm}
    \|u\|_{\frH_{neu}^{\gamma+2}(p,\theta,\tau)} &:= \|u\|_{\widetilde{\bH}_{p,\theta}^{\gamma+2,(2)}(\tau)} + \|u_0\|_{U_{p,\theta}^{\gamma+2}} + \|\bD u\|_{\bH_{p,\theta}^{\gamma}(\tau)} + \|\bS u\|_{{\bH}_{p,\theta}^{\gamma+1,(1)}(\tau,l_2)}.
\end{align}

We then define a solution to \eqref{maineq} in the space $\frH_{neu}^{\gamma+2}(p,\theta,\tau)$. We say that $u\in \frH_{neu}^{\gamma+2}(p,\theta,\tau)$ is a solution to \eqref{maineq} if $u$ satisfies
\begin{equation*}
    du = \left(a^{ij}(t)D_{ij}u+f(u)\right)dt + \left(\sigma^{ik}(t)D_iu+g^k(u)\right)dw_t^k
\end{equation*}
in the sense of distributions.
\end{definition}

\subsection{Main results} \label{sec_res}

In this subsection, we introduce our main results. First, we present the assumptions on the coefficients and semi-linear conditions for $f$ and $g$.

\begin{assumption} \label{as_coeff}
    There exist constants $\delta_0,\delta_1\in(0,1]$ such that for any $\omega\in\Omega$, $t\in(0,T)$, and $\xi\in \bR^d$,
    \begin{equation*}
        \delta_0|\xi|^2 \leq \delta_1 a^{ij}(t) \xi_i \xi_j \leq \overline{a}^{ij}(t) \xi_i \xi_j \leq a^{ij}(t) \xi_i \xi_j \leq \delta_0^{-1}|\xi|^2,
    \end{equation*}
    where
    \begin{equation*}
        \overline{a}^{ij}(t):= a^{ij}(t) - \alpha^{ij}(t), \quad \alpha^{ij}(t):= \frac{1}{2}\sigma^{ik}(t) \sigma^{jk}(t).
    \end{equation*}
\end{assumption}

\begin{assumption} \label{as_semi}
$(i)$ For any $u \in \widetilde{H}_{p,\theta}^{\gamma+2,(2)}$, $f(t,x,u)$ and $g(t,x,u)$ are predictable functions taking values in $H_{p,\theta}^{\gamma}$ and $H_{p,\theta}^{\gamma+1,(1)}(l_2)$, respectively. Moreover,
$$
f(0)\in \bH_{p,\theta}^{\gamma}(\tau), \quad g(0) \in \bH_{p,\theta}^{\gamma+1,(1)}(\tau,l_2).
$$

$(ii)$ For any $\varepsilon>0$, there exists a constant $N_\varepsilon>0$ such that for any $\omega,t$, and $u,v\in \widetilde{H}_{p,\theta}^{\gamma+2,(2)}$,
\begin{align} \label{eq4221624}
&\|f(t,\cdot,u)-f(t,\cdot,v)\|_{H_{p,\theta}^{\gamma}}^p  + \|g(t,\cdot,u)-g(t,\cdot,v)\|_{{H}_{p,\theta}^{\gamma+1,(1)}(l_2)}^p \nonumber
\\
&\leq \varepsilon \left( \|D_x^2u-D_x^2v\|^p_{H_{p,\theta}^{\gamma}} + \|x_d^{-1}D_du-x_d^{-1}D_dv\|^p_{H_{p,\theta}^{\gamma+1}} \right)+ N_\varepsilon \|u-v\|_{H_{p,\theta}^{\gamma}}^p.
\end{align}
\end{assumption}

\begin{remark}
Clearly, by the definition of $\widetilde{H}_{p,\theta}^{\gamma+2,(2)}$, the right-hand side of \eqref{eq4221624} is finite.
\end{remark}

\begin{example}
We present simple examples in both the linear and nonlinear settings. Further examples can be constructed following \cite[Section 5]{Kry99}. We also refer the reader to \cite{KKL19}.

$(i)$ The linear case is included by taking
    \begin{equation*}
        f(u) := b_i D_{i}u + cu + f_0, \quad g(u) := \nu u + g_0,
    \end{equation*}
    where $b_i$ and $c$ are bounded predictable functions on $(0,\tau]$, $\nu$ is an $l_2$-valued bounded predictable function on $(0,\tau]$, $f_0\in \bH_{p,\theta}^{\gamma}(\tau)$, and $g_0\in {\bH}_{p,\theta}^{\gamma+1,(1)}(\tau,l_2)$. To show $f$ and $g$ satisfy Assumption \ref{as_semi}, it suffices to prove \eqref{eq4221624}. By a straightforward calculation,
    \begin{align*}
        &\|f(t,\cdot,u)-f(t,\cdot,v)\|_{H_{p,\theta}^{\gamma}}^p + \|g(t,\cdot,u)-g(t,\cdot,v)\|_{H_{p,\theta}^{\gamma+1,(1)}(l_2)}^p \nonumber
        \\
        &\leq N \left(\|D_xu-D_xv\|^p_{H_{p,\theta}^{\gamma}} + \|u-v\|^p_{H_{p,\theta}^{\gamma}} \right).
    \end{align*}
    Thus, it remains to show that for $w\in \widetilde{H}_{p,\theta}^{\gamma+2,(2)}$,
    \begin{equation} \label{eq5131329}
        \|D_xw\|^p_{H_{p,\theta}^{\gamma}} \leq N\|w\|_{H_{p,\theta}^{\gamma}}^{p/2} \left(\|x_d^{-1}D_dw\|_{H_{p,\theta}^{\gamma+1}}^{p/2} + \|D_x^2w\|_{H_{p,\theta}^{\gamma}}^{p/2} \right).
    \end{equation}
By \eqref{eq5071804} with $(\kappa,\theta_1,\theta_2,\gamma_1,\gamma_2) = (1/2,\theta-p,\theta+p,\gamma+1,\gamma-1)$,
\begin{align*}
    \|D_dw\|^p_{H_{p,\theta}^{\gamma}} \leq N \|x_d^{-1}D_dw\|_{H_{p,\theta}^{\gamma+1}}^{p/2} \|x_dD_dw\|_{H_{p,\theta}^{\gamma-1}}^{p/2}.
\end{align*}
Thus, when $D_x=D_d$, \eqref{eq5131329} follows from \eqref{equiv_der}. For $D_{x'}w$, we use Gagliardo-Nirenberg interpolation inequality, whose classical version states that
\begin{equation*}
    \|D_iw\|_{L_p(\bR^d)} \leq N\|w\|_{L_p(\bR^d)}^{1/2}\|D_i^2w\|_{L_p(\bR^d)}^{1/2}.
\end{equation*}
By considering $(1-\Delta)^{\gamma/2}w$ instead of $w$, the above inequality also holds in $H_p^{\gamma}(\bR^d)$. Using this and H\"older's inequality,
    \begin{align*}
        \|D_iw\|^p_{H_{p,\theta}^{\gamma}} &= \sum_{n\in\bZ} e^{n\theta} \|(D_iw)(e^n\cdot)\zeta\|_{H_p^{\gamma}}^p = \sum_{n\in\bZ} e^{n(\theta-p)} \|D_i(w(e^n\cdot)\zeta)\|_{H_p^{\gamma}}^p 
        \\
        &\leq N \sum_{n\in\bZ} e^{n(\theta-p)} \left(\|w(e^n\cdot)\zeta\|_{H_p^{\gamma}}^{p/2} \times \|D_i^2(w(e^n\cdot)\zeta)\|_{H_p^{\gamma}}^{p/2} \right)
        \\
        &\leq N \left(\sum_{n\in\bZ} e^{n\theta} \|w(e^n\cdot)\zeta\|_{H_p^{\gamma}}^{p}\right)^{1/2} \left(\sum_{n\in\bZ} e^{n(\theta-2p)} \|D_i^2(w(e^n\cdot)\zeta)\|_{H_p^{\gamma}}^{p}\right)^{1/2}
        \\
        &= N \|D_i^2w\|^{p/2}_{H_{p,\theta}^{\gamma}}\|w\|_{H_{p,\theta}^{\gamma}}^{p/2}.
    \end{align*}
    This leads to \eqref{eq5131329} when $D_x=D_i$.

    $(ii)$
   We also give a nonlinear example. Suppose that $\gamma>d/p$, and take a fixed nonrandom function $f_0=f_0(x) \in H_{p,\theta}^{\gamma}$. Let
\begin{equation*}
f(u)(x) := f_0(x) \sup_{y\in\bR_+^d} |y_d^{\theta/p}u(y)|.
\end{equation*}
By the Sobolev embedding in $H_{p,\theta}^{\gamma}$ (see \eqref{eq5250100}), for any $u,v\in\widetilde{H}_{p,\theta}^{\gamma+2,(2)}$,
\begin{align*}
\|f(u)-f(v)\|_{H_{p,\theta}^{\gamma}} \leq \|f_0\|_{H_{p,\theta}^{\gamma}} \sup_{y\in\bR_+^d}|y_d^{\theta/p}(u(y)-v(y))| \leq N\|f_0\|_{H_{p,\theta}^{\gamma}} \|u-v\|_{H_{p,\theta}^{\gamma}}.
\end{align*}
Thus, \eqref{eq4221624} holds with $N_\varepsilon$ independent of $\varepsilon$. Moreover, $f(0)=0$, and $f$ is a continuous map from $H_{p,\theta}^{\gamma}$ to itself. Hence, Assumption \ref{as_semi} is satisfied.
\end{example}

We state the first main result about the solvability of equation \eqref{maineq}.

\begin{theorem} \label{thm_main}
    Let $\gamma\in\bR$, $p\in[2,\infty)$, $T\in(0,\infty)$, and $\tau\leq T$ be a bounded stopping time.
    Suppose that Assumptions \ref{as_coeff} and \ref{as_semi} hold, and we assume that
    \begin{equation} \label{theta}
        d-1+p\left( 1- \frac{1}{p(1-\delta_1)+\delta_1} \right)<\theta < d-1+p.
    \end{equation}
    Assume that $u_0\in U_{p,\theta}^{\gamma+2}$. Then there is a unique solution $u\in  \frH_{neu}^{\gamma+2}(p,\theta,\tau)$ to \eqref{maineq}. Moreover, we have
    \begin{align} \label{main_est}
    \|u\|_{\frH_{neu}^{\gamma+2}(p,\theta,\tau)} &\leq N\left( \|u_0\|_{U_{p,\theta}^{\gamma+2}} + \|f(0)\|_{\bH_{p,\theta}^{\gamma}(\tau)} + \|g(0)\|_{{\bH}_{p,\theta}^{\gamma+1,(1)}(\tau,l_2)} \right),
    \end{align}
where $N$ depends only on $d,\gamma,p,\theta,\delta_0,\delta_1,T$, and the function $N_\varepsilon$.
\end{theorem}

\begin{remark} \label{rem5250111}
$(i)$ The range of $\theta$ is the same as that for the Dirichlet problem considered in \cite{Kry_halfspace}. The condition \eqref{theta} is satisfied whenever $d-2+p\leq\theta<d-1+p$. In the special case where $\sigma^{ik}=0$ for all $i$ and $k$, one can take $\delta_1=1$ in \eqref{theta}, which yields the larger range $d-1<\theta<d-1+p$. This coincides with the admissible range for the deterministic Dirichlet problem. On the other hand, for the stochastic Dirichlet problem with general $\sigma^{ik}$, it is known that the full range $d-1<\theta<d-1+p$ cannot be covered, although the sharp admissible range remains an open problem (see \cite{K03,Kry_halfspace}).

$(ii)$
Let $u$ be a solution taken from Theorem \ref{thm_main}.
Suppose further that $\gamma>d/p-1$ and $\theta<p$. Since
$x_d^{-1}D_du(t,\cdot)\in H_{p,\theta}^{\gamma+1}$, \eqref{eq5250100} implies that
\begin{equation*}
x_d^{\theta/p-1}D_du\in C(\mathbb R_+^d).
\end{equation*}
Thus, due to $\theta/p-1<0$, $D_du$ admits a continuous extension to
$\overline{\mathbb R_+^d}$ satisfying
\begin{equation*}
D_du(t,x',0)=0.
\end{equation*}
Consequently, if $\gamma>d/p-1$ and $\theta<p$, then, almost surely, the solution is sufficiently regular and the Neumann boundary condition is satisfied in the classical sense.
\end{remark}

In the remainder of this section, we present space-time Sobolev H\"older embedding results. The proofs are provided in Section \ref{sec_embed}.

We begin with the results near the boundary $x_d=0$. We first present a regularity result for $I_Ru := u\eta_R$, where $\eta_R=\eta_R(x_d)\in C^\infty(\bR^d_+)$ is a one-dimensional cutoff function such that $\eta_R=1$ on $\{0<x_d<R\}$ and $\eta_R=0$ on $\{x_d>2R\}$.

\begin{theorem} \label{thm_bdry}
    Let $u\in \frH_{neu}^{\gamma+2}(p,\theta,\tau)$ be the solution taken from Theorem \ref{thm_main}. Let $1/2>\beta>\alpha>1/p$, and assume that $\gamma_0+2-\frac{\theta_0}{p}-2\beta>0$ is not an integer, where $\theta_0 := \max\{\theta,d\}$ and $\gamma_0 := \min\{\gamma,0\}$. Then for any $R>0$,
\begin{align} \label{eq5251945}
    \bE \|I_Ru\|_{C^{\alpha-1/p}\left([0,\tau];C^{\gamma_0+2-\frac{\theta_0}{p}-2\beta}(\overline{\bR^d_+})\right)}^p <\infty.
\end{align}
In the special case when $\gamma=0$ and $\theta\geq d$, we have the following global regularity result
\begin{align} \label{eq5271711}
    \bE \|u\|_{C^{\alpha-1/p}\left([0,\tau];C^{2-\frac{\theta}{p}-2\beta}(\overline{\bR^d_+})\right)}^p <\infty.
\end{align}
\end{theorem}

Next, we present the embedding result away from the boundary.

\begin{theorem} \label{thm_int}
Let $u\in \frH_{neu}^{\gamma+2}(p,\theta,\tau)$ be the solution taken from Theorem \ref{thm_main}. Let $1/2>\beta>\alpha>1/p$, and assume that $\gamma+2-\frac{d}{p}-2\beta>0$ is not an integer. Then for any $\chi \in C_c^\infty(\bR^d_+)$,
\begin{align} \label{eq5251946}
    \bE \|\chi u\|_{C^{\alpha-1/p}\left([0,\tau];C^{\gamma+2-\frac{d}{p}-2\beta}(\overline{\bR^d_+})\right)}^p <\infty.
\end{align}
\end{theorem}

\begin{remark} \label{rem1013}
We elaborate on Theorems \ref{thm_bdry} and \ref{thm_int}.

$(i)$ Assume that $\sigma^{ik}\equiv0$ so that we can take $\delta_1=1$ and $\theta=d$ in \eqref{theta}. Let us fix $\varepsilon>0$ and $\gamma+2\in (0,2]$. Then by taking both $\alpha$ and $\beta$ sufficiently close to $1/p$, and $p$ sufficiently large, \eqref{eq5251945} and \eqref{eq5251946} yield that for any $x_0\in \bR^d_+$,
\begin{equation*}
\sup_{t\in [0,\tau]} \|u(t,\cdot)\|_{C^{\gamma+2-\varepsilon}(B_1(x_0)\cap \overline{\bR^d_+})} < \infty \, (a.s.).
\end{equation*}
This is notably different from the Dirichlet case. Indeed, even when $g\in C_c^\infty([0,\tau)\times \overline{\bR^d_+})$ is a nonrandom function independent of $u$, \cite[Theorem 3.3]{Kry_halfspace} and \eqref{eq5250100} below yield, at best, the spatial regularity
\begin{equation*}
    u(t,\cdot) \in C^{1-\varepsilon}(B_1(x_0)\cap \overline{\bR^d_+}).
\end{equation*}
 This phenomenon is related to the possible failure of $D_x^2u\in \bL_{p,d}$ (see \cite[Section 1]{K94}).

 $(ii)$ Next, we impose no restriction on $\delta_1$. We consider the particular choice $\theta=d-2+p$, which is allowed by Remark \ref{rem5250111}. Assume that $\gamma+2\in (1,2]$. As in $(i)$, by taking both $\alpha$ and $\beta$ sufficiently close to $1/p$, and $p$ sufficiently large, we get
 \begin{equation*}
\sup_{t\in [0,\tau]} \|u(t,\cdot)\|_{C^{\gamma+1-\varepsilon}(B_1(x_0)\cap \overline{\bR^d_+})} < \infty \, (a.s.).
\end{equation*}
We emphasize that this type of boundary regularity remains valid even in the presence of the stochastic term $\sigma^{ik}$. In particular, when $\gamma+2=2$, we deduce that $u(t,\cdot) \in C^{1-\varepsilon}$ near the boundary.

This should be compared with the Dirichlet case, where the stochastic coefficients $\sigma^{ik}$ cause a loss of boundary regularity. In particular, \cite{K03} shows that, in the one-dimensional case $d=1$, the optimal regularity near the boundary cannot reach $C^{1-\varepsilon}$.
More precisely, for any $\varepsilon>0$, one can construct an equation with nonzero $\sigma^{ik}=\sigma^{1k}$ whose solution has boundary H\"older regularity of order less than $\varepsilon$.
The same phenomenon is expected in general dimensions, and weaker regularity estimates are presented in \cite{M19}.
\end{remark}

\section{Some properties of function spaces} \label{sec_funct}

\subsection{Homogeneous weighted Sobolev spaces}

First, we state some basic properties of the homogeneous weighted Sobolev space $H_{p,\theta}^{\gamma}$.

\begin{lemma} \label{lem4161836}
Let $p\in(1,\infty)$, and $\gamma,\theta\in\bR$.

(i) The space $H_{p,\theta}^\gamma$ is a Banach space, and the set $C_c^\infty(\bR^d_+)$ is dense in this space.

(ii) 
The following norms are equivalent: for any $\delta\in\bR$,
\begin{equation} \label{eq5062159}
    \|u\|_{H_{p,\theta}^\gamma} \approx \|x_d^{-\delta} u\|_{H_{p,\theta+\delta p}^\gamma}
\end{equation}
and
\begin{equation} \label{equiv_der}
    \|u\|_{H_{p,\theta}^{\gamma}} \approx \left(\|u\|_{H_{p,\theta}^{\gamma-1}} + \|x_dD_xu\|_{H_{p,\theta}^{\gamma-1}}\right).
\end{equation}
In addition, for any $\mu\leq \gamma$,
\begin{equation} \label{eq5062157}
    \|u\|_{H_{p,\theta}^{\gamma}} \approx \left(\|u\|_{H_{p,\theta}^{\mu}} + \|x_dD_xu\|_{H_{p,\theta}^{\gamma-1}}\right).
\end{equation}

(iii) (Hardy's inequality)
If $\theta\neq d-1$ and $u\in H_{p,\theta}^{\gamma}$, then
\begin{equation*}
    \|u\|_{H_{p,\theta}^{\gamma}} \leq N(d,\gamma,p,\theta) \|x_dD_xu\|_{H_{p,\theta}^{\gamma-1}}.
\end{equation*}
Moreover, if $\theta\neq d-1+p$ and $u\in H_{p,\theta}^{\gamma}$, then
\begin{equation*}
    \|u\|_{H_{p,\theta}^{\gamma}} \leq N(d,\gamma,p,\theta) \|D_x(x_du)\|_{H_{p,\theta}^{\gamma-1}}.
\end{equation*}

(iv) Assume
$\gamma-\frac{d}{p} \geq n+\delta$ for  $n\in\bN_0$ and $\delta\in(0,1)$. Then for all $k\leq n$,
\begin{equation} \label{eq5250100}
    |x_d^{k+\frac{\theta}{p}} D_x^k u|_{C(\bR^d_+)}+[x_d^{n+\frac{\theta}{p}+\delta} D_x^{n}u]_{C^{\delta}(\bR^d_+)}\leq N(d,\gamma,p,\theta)\|u\|_{H_{p,\theta}^{\gamma}}.
\end{equation}

(v) (Interpolation)
Let $\kappa\in(0,1)$, $\gamma_i,\theta_i\in\bR$, $i=1,2$, and
\begin{equation*}
    \theta=\kappa\theta_1+(1-\kappa)\theta_2, \quad \gamma=\kappa\gamma_1 + (1-\kappa)\gamma_2.
\end{equation*}
Then we have
\begin{align} \label{eq5071804}
H_{p,\theta}^\gamma=[H_{p,\theta_1}^{\gamma_1},H_{p,\theta_2}^{\gamma_2}]_\kappa,
\end{align}
where $[A,B]_\kappa$ is the complex interpolation space of $A$ and $B$.

(vi) (Duality)
Let
\begin{equation*}
1/p+1/p'=1, \quad \theta/p+\theta'/p' = d.
\end{equation*}
Then the dual space of $H_{p,\theta}^\gamma$ is $H_{p',\theta'}^{-\gamma}$. Moreover, for $u\in H_{p,\theta}^\gamma$, the pairing $(u,\phi)$, initially defined on test functions $\phi\in C_c^\infty(\bR^d_+)$, can be extended continuously to all of  $H_{p',\theta'}^{-\gamma}$.

(vii)
The assertions in (i)-(v) remain valid when $H_{p,\theta}^\gamma$ is replaced by $H_{p,\theta}^\gamma(l_2)$. In particular, the set of all $l_2$-valued functions $u=(u^1,u^2,\cdots)$ such that $u^k=0$ for all sufficiently large $k$, and each component $u^k \in C_c^\infty(\bR^d_+)$, is dense in $H_{p,\theta}^\gamma(l_2)$.

\end{lemma}

\begin{proof}
See Sections 1-3 in \cite{Kry_some} for the results on $H_{p,\theta}^{\gamma}$. We remark that the same proofs go through for $H_{p,\theta}^\gamma(l_2)$.
\end{proof}

In general, interpolation does not commute with intersections of Banach spaces. However, for the spaces $H_{p,\theta}^\gamma$, such a commutation property does hold. We note that for two Banach spaces $X$ and $Y$ continuously embedded into a common Hausdorff topological vector space, $X \cap Y$ is equipped with the norm
\begin{equation*}
\|u\|_{X \cap Y}:=\|u\|_{X}+\|u\|_{Y}.
\end{equation*}

\begin{lemma} \label{lem_inter}
Let $n\in[0,\infty)$, $\nu\in\bR$, $\kappa\in(0,1)$, $p\in(1,\infty)$, and let $\gamma_i,\theta_i\in\bR$ for $i=1,2$. Define
\begin{equation*}
\theta := \kappa\theta_1 + (1-\kappa)\theta_2, \quad \gamma := \kappa \gamma_1 + (1-\kappa) \gamma_2.
\end{equation*}
Then we have
\begin{equation} \label{eq9031054}
\left[ H^{\gamma_1}_{p,\theta_1} \cap H^{\gamma_1-n}_{p,\theta_1+\nu p}, H^{\gamma_2}_{p,\theta_2} \cap H^{\gamma_2-n}_{p,\theta_2+\nu p} \right]_\kappa = H^{\gamma}_{p,\theta}\cap H^{\gamma-n}_{p,\theta+\nu p}.
\end{equation}
\end{lemma}

\begin{proof} 
\textit{Step 1.} We claim that for $\gamma_0\in\bR$, $m\in\bZ$, and $v\in H_p^{\gamma_0}$, there exists $N=N(d,p,n)$ such that
\begin{equation} \label{eq9031211}
\|\Lambda_m v\|_{H^{\gamma_0}_{p}} \leq \|v\|_{H^{\gamma_0}_{p}} + e^{m\nu}\|v\|_{H^{\gamma_0-n}_{p}} \leq N \|\Lambda_m v\|_{H^{\gamma_0}_{p}},
\end{equation}
where $\Lambda_m:=1+e^{m\nu}(1-\Delta)^{-n/2}$. Here,
\begin{equation*}
\Lambda_m v = \cF^{-1} \left\{\left(1 + e^{m\nu}(1+|\xi|^2)^{-n/2}\right)\cF v\right\}.
\end{equation*}
 The first inequality follows from $\|(1-\Delta)^{-n/2}v\|_{H^{\gamma_0}_{p}} = \|v\|_{H^{\gamma_0-n}_{p}}$ and the triangle inequality. 
 Thus, it remains to prove the second inequality. Since the operators $(1-\Delta)^{\gamma_0/2}$ and $\Lambda_m$ commute, it suffices to consider the case when $\gamma_0=0$. If $n=0$, then \eqref{eq9031211} is obvious. Thus, we assume that $n>0$. Note that
 \begin{equation} \label{eq9032336}
     v = \frac{1}{1+e^{m\nu}(1-\Delta)^{-n/2}} \Lambda_m v
 \end{equation}
 and
  \begin{equation} \label{eq9032337}
     e^{m\nu}(1-\Delta)^{-n/2} v = \frac{e^{m\nu}(1-\Delta)^{-n/2}}{1+e^{m\nu}(1-\Delta)^{-n/2}} \Lambda_m v.
 \end{equation}
 Let
\begin{equation*}
    h_1(t):=\frac{1}{1+t^{-n}}, \quad h_2(t):=\frac{t^{-n}}{1+t^{-n}}.
\end{equation*}
 Then one can check that $\sup_{t>0}|t^kh_i^{(k)}(t)|<\infty$ for every $k\in\bN_0$. Using this, we obtain that for any multi-index $\alpha$,
\begin{equation*}
    |\xi|^{|\alpha|} \left|D^\alpha_\xi \left[h_i\left(e^{-m\nu/n}(1+|\xi|^2)^{1/2}\right)\right]\right|\leq N,
\end{equation*}
where $N=N(d,n,\alpha)$ is independent of $m$ and $\nu$. Hence by the Mikhlin-H\"ormander multiplier theorem (see e.g.\ \cite[Theorem 6.2.7]{G14}),
\begin{equation*}
    \frac{1}{1+e^{m\nu}(1-\Delta)^{-n/2}}, \quad \frac{e^{m\nu}(1-\Delta)^{-n/2}}{1+e^{m\nu}(1-\Delta)^{-n/2}} 
\end{equation*}
are bounded operators on $L_p$. Thus, the second inequality of \eqref{eq9031211} follows from \eqref{eq9032336} and \eqref{eq9032337}.

 \textit{Step 2.} For $\gamma_0,\theta_0\in\bR$, we denote by $l_p^{\theta_0}(H_p^{\gamma_0})$ the space of sequences $v=(v_m)_{m\in\bZ}$ with values in $H_p^{\gamma_0}$ equipped with the norm
\begin{equation*}
\|v\|_{l_p^{\theta_0}(H_p^{\gamma_0})}^p := \sum_{m\in\bZ} e^{m\theta_0} \|v_m\|_{H_p^{\gamma_0}}^p.
\end{equation*}
Let $\zeta\in C_c^\infty(\bR_+)$ be taken from \eqref{defH}. By replacing $\zeta$ with $\zeta(x_d)/\sum_{m\in\bZ}\zeta(e^mx_d)$, we may assume that
\begin{equation} \label{eq9040101}
\sum_{m\in\bZ}\zeta(e^{m}x_d)=1, \qquad x_d>0.
\end{equation}
Let $\eta\in C_c^\infty(\bR_+)$ such that $\eta=1$ on $\text{supp}(\zeta)$, and define
\begin{equation*}
Su := \left(\Lambda_m(u(e^m\cdot)\zeta)\right)_{m\in\bZ}, \quad
Rv := \sum_{m\in\bZ}\eta(e^{-m} x_d)\,(\Lambda_m^{-1}v_m)(e^{-m}\cdot).
\end{equation*}
By Step 1, both $\Lambda_m$ and $\Lambda_m^{-1}=(1+e^{m\nu}(1-\Delta)^{-n/2})^{-1}$ are bounded on $H^{\gamma_0}_p$ for each $m\in\bZ$. Hence, for $u\in H^{\gamma_0}_{p,\theta_0}$, $Su$ is a well-defined sequence in $H^{\gamma_0}_p$, and for $v\in l^{\theta_0}_p(H^{\gamma_0}_p)$, $Rv$ is well defined since the sum is locally finite. 

In this step, we show that
\begin{equation} \label{eq9040106}
S: H^{\gamma_0}_{p,\theta_0}\cap H^{\gamma_0-n}_{p,\theta_0+\nu p} \to l^{\theta_0}_p(H^{\gamma_0}_p), \quad R: l^{\theta_0}_p(H^{\gamma_0}_p)\to H^{\gamma_0}_{p,\theta_0}\cap H^{\gamma_0-n}_{p,\theta_0+\nu p}
\end{equation}
are bounded operators. For $S$, by \eqref{defH} and \eqref{eq9031211},
\begin{align*}
    \|Su\|_{l^{\theta_0}_p(H^{\gamma_0}_p)}^p &= \sum_{m\in\bZ} e^{m\theta_0} \|\Lambda_m(u(e^m\cdot)\zeta)\|_{H_p^{\gamma_0}}^p
    \\
    &\leq N \sum_{m\in\bZ} e^{m\theta_0} \left(\|u(e^m\cdot)\zeta\|_{H_p^{\gamma_0}}^p + e^{m\nu p}\|u(e^m\cdot)\zeta\|_{H_p^{\gamma_0-n}}^p\right)
    \\
    &\leq N \|u\|_{H^{\gamma_0}_{p,\theta_0}\cap H^{\gamma_0-n}_{p,\theta_0+\nu p}}^p.
\end{align*}
Now we deal with $R$. We claim that
\begin{equation*}
\|Rv\|_{H^{\gamma_0}_{p,\theta_0} \cap H^{\gamma_0-n}_{p,\theta_0+\nu p}}\leq N \|v\|_{l^{\theta_0}_p(H^{\gamma_0}_p)}.
\end{equation*}
Let $w=(w_m)_{m\in\bZ} := (\Lambda_m^{-1}v_m)_{m\in\bZ}$. By \eqref{eq9031211},
\begin{equation*}
\|w\|^p_{l^{\theta_0}_p(H^{\gamma_0}_p)} + \|w\|^p_{l^{\theta_0+\nu p}_p(H^{\gamma_0-n}_p)} \leq N\|v\|^p_{l^{\theta_0}_p(H^{\gamma_0}_p)}.
\end{equation*}
Thus, it suffices to show that for any $\tilde{\gamma}, \tilde{\theta} \in \bR$,
\begin{equation} \label{eq9040104}
\left\| \sum_{m\in\bZ}\eta(e^{-m}x_d) w_m(e^{-m}\cdot) \right\|_{H_{p,\tilde{\theta}}^{\tilde{\gamma}}} \leq N \|w\|_{l_p^{\tilde{\theta}}(H_p^{\tilde{\gamma}})}.
\end{equation}
Since $\zeta,\eta\in C_c^\infty(\bR_+)$, there is $n_0\in\bN$ so that 
\begin{equation*}
\zeta(x_d) \left( \sum_{m\in\bZ}\eta(e^{k-m}x_d)w_m(e^{k-m}x) \right) = \sum_{|j|\leq n_0} \zeta(x_d) \eta(e^{j}x_d) w_{k-j}(e^{j}x).
\end{equation*}
Since the multiplication by $\zeta\,\eta(e^{j}\,\cdot)$ and the dilation $w\to w(e^{j}\cdot)$, $|j|\leq n_0$, are bounded on $H^{\tilde\gamma}_p$ (see e.g. \cite[Corollary 4.2.2 and Theorem 4.3.2]{Triebel2}), by \eqref{defH},
\begin{equation*}
\left\| \sum_{m\in\bZ}\eta(e^{-m}x_d) w_m(e^{-m}\cdot) \right\|_{H_{p,\tilde{\theta}}^{\tilde{\gamma}}}^p
\leq N \sum_{|j|\leq n_0} \sum_{k\in\bZ} e^{k\tilde{\theta}} \|w_{k-j}\|_{H_p^{\tilde{\gamma}}}^p
= N\sum_{|j| \leq n_0} e^{j\tilde{\theta}} \|w\|^p_{l_p^{\tilde{\theta}}(H_p^{\tilde{\gamma}})}.
\end{equation*}
Thus, we obtain \eqref{eq9040104}, which proves the claim.

\textit{Step 3.}
By \cite[Theorem 1.18.1]{T78} and $[H_p^{\gamma_1},H_p^{\gamma_2}]_\kappa=H_p^{\gamma}$, we have
\begin{equation} \label{eq9040105}
\left[l^{\theta_1}_p(H^{\gamma_1}_p),\ l^{\theta_2}_p(H^{\gamma_2}_p)\right]_\kappa = l^{\theta}_p(H^{\gamma}_p).
\end{equation}
Here and below, $[A_0,A_1]_\kappa$ denotes the complex interpolation space with the convention for $\kappa$ reversed from that in \cite{T78}; that is, $[A_0,A_1]_\kappa$ in our notation corresponds to $[A_0,A_1]_{1-\kappa}$ in \cite[Section 1.9]{T78}.
Due to \eqref{eq9040101},
\begin{equation*}
RSu = \sum_{m\in\bZ}\eta(e^{-m}x_d)\zeta(e^{-m}x_d)u = \sum_{m\in\bZ}\zeta(e^{-m}x_d) u = u.
\end{equation*}
Thus, by \eqref{eq9040106} with $(\gamma_0,\theta_0)=(\gamma_i,\theta_i)$, $i=1,2$, \eqref{eq9040105}, and \cite[Theorem 1.2.4]{T78}, the operator $S$ is an isomorphism from
\begin{equation*}
\left[H^{\gamma_1}_{p,\theta_1}\cap H^{\gamma_1-n}_{p,\theta_1+\nu p}, H^{\gamma_2}_{p,\theta_2}\cap H^{\gamma_2-n}_{p,\theta_2+\nu p}\right]_\kappa
\end{equation*}
onto $SR\left(l^{\theta}_p(H^{\gamma}_p)\right)$. Since $RS$ is the identity operator,
\begin{equation*}
    \left[H^{\gamma_1}_{p,\theta_1}\cap H^{\gamma_1-n}_{p,\theta_1+\nu p}, H^{\gamma_2}_{p,\theta_2}\cap H^{\gamma_2-n}_{p,\theta_2+\nu p}\right]_\kappa = R\left(l^{\theta}_p(H^{\gamma}_p)\right).
\end{equation*}
Finally, by \eqref{eq9040106} with $(\gamma_0,\theta_0)=(\gamma,\theta)$, we get \eqref{eq9031054}. The lemma is proved.
\end{proof}

\begin{remark}
    Lemma \ref{lem_inter} is a generalization of Lemma \ref{lem4161836} $(v)$. Indeed, \eqref{eq5071804} is recovered by taking $n=0$ and $\nu=0$ in \eqref{eq9031054}.
\end{remark}

Next, we introduce an isomorphism between $H_{p,\theta}^{\gamma+2}$ and $H_{p,\theta}^{\gamma}$.
Let us consider the degenerate operator
\begin{equation} \label{eq5090021}
    \cL_{b,c}u := x_d^2 \Delta u + bx_d D_du - cu, \quad b,c\in \bR.
\end{equation}

\begin{lemma} \label{lem4282020}
    Let $b,c\in\bR$. Then there is $c_0=c_0(d,p,\theta,\gamma,b)>0$ such that if $c>c_0$, then $\cL_{b,c}$ is bounded and invertible from $H_{p,\theta}^{\gamma+2}$ to $H_{p,\theta}^{\gamma}$.
\end{lemma}

\begin{proof}
    See \cite[Theorems 4.4 and 5.2]{Kry_some}.
\end{proof}

\subsection{General weighted Sobolev spaces}

We prove several properties of $H_{p,\theta}^{\gamma,(n)}$ corresponding to Lemma \ref{lem4161836}, dividing the proof into several parts. Recall that
\begin{equation*}
    H_{p,\theta}^{\gamma,(n)} = \{u: u, D_xu, \cdots, D_x^n u\in H_{p,\theta}^{\gamma-n}\}, \quad n\in \bN_0.
\end{equation*}

\begin{lemma}
Let $p\in(1,\infty)$, $n\in\bN_0$, and $\gamma,\theta\in\bR$.
Then the space $H_{p,\theta}^{\gamma,(n)}$ is a Banach space.
\end{lemma}

\begin{proof}
    Suppose that $u_l$ is a Cauchy sequence in $H_{p,\theta}^{\gamma,(n)}$. Since $H_{p,\theta}^{\gamma-n}$ is a Banach space, we can put $v^r:= \lim_{l\to\infty} D_x^ru_l$ for each $r=0,1,\dots,n$. Obviously,
    \begin{equation*}
        (v^r,\phi) = \lim_{l\to\infty} (D_x^ru_l,\phi) = (-1)^r \lim_{l\to\infty} (u_l,D_x^r\phi) = (-1)^r (v^0,D_x^r\phi),
    \end{equation*}
    which implies that $D_x^rv^0 = v^r \in H_{p,\theta}^{\gamma-n}$. Thus, $v^0=\lim_{l\to\infty}u_l \in H_{p,\theta}^{\gamma,(n)}$. The lemma is proved.
\end{proof}

Next, we aim to formulate the trace theorem for the weighted Sobolev spaces. By a trace operator, we usually mean a map that evaluates boundary values, such as $u \mapsto u(\cdot,0)$. 
For higher regularity, it is natural to expect that the boundary values of derivatives should also be taken into account. More precisely, for $H_{p,\theta}^{\gamma,(n)}$, we consider the trace operator of the form
\begin{equation} \label{eq5251711}
T_ru(x'):=\left(u(x',0),\ldots,D_d^{r-1}u(x',0)\right),
\end{equation}
where $r$ is determined by $\theta$, $\gamma$, and $n$.
To define this operator on the entire space, we first establish the density of sufficiently regular functions and then extend the operator by continuity.

We first state a known result for the special case $\gamma=n$. We recall the Slobodeckij space of order $s\in (0,\infty)\setminus\bN$, denoted by $W_p^s(\bR^{d-1})$, which is the set of functions $u\in L_p(\bR^{d-1})$ such that
\begin{equation*}
    \sum_{r\leq \lfloor s \rfloor}\|D_x^ru\|_{L_p(\bR^{d-1})} + [D_x^{\lfloor s \rfloor}u]_{W_p^{s-\lfloor s \rfloor}(\bR^{d-1})}<\infty,
\end{equation*}
where $\lfloor s \rfloor$ is the greatest integer less than or equal to $s$, and
\begin{equation*}
    [v]_{W_p^{\sigma}(\bR^{d-1})}:= \left( \int_{\bR^{d-1}}\int_{\bR^{d-1}} \frac{|v(x')-v(y')|^p}{|x'-y'|^{d-1+\sigma p}} dy'dx' \right)^{1/p}, \quad \sigma\in(0,1).
\end{equation*}

\begin{lemma} \label{lem5061739}
    Suppose that $p\in(1,\infty)$, $n\in\bN$, and $\theta\in(d-1,\infty)$.

    (i) The set $C_c^\infty(\overline{\bR^d_+})$ is dense in $H_{p,\theta}^{n,(n)}$.

    (ii) Assume that $\theta\in(d-1,d-1+p)$. For each $k\in\bN_0$, the operator $T_n$ extends to a bounded linear operator from $H_{p,\theta+kp}^{n+k,(n+k)}$ onto
    \begin{equation} \label{eq5061624}
        \prod_{s=0}^{n-1}W_{p}^{n-\frac{\theta-d+1}{p}-s}(\bR^{d-1}).
    \end{equation}
    Moreover, there is an extension operator $S_{n,k}$ from \eqref{eq5061624} to $H_{p,\theta+kp}^{n+k,(n+k)}$ such that $T_nS_{n,k}=I$, where $I$ denotes the identity operator.
\end{lemma}

\begin{proof}
    $(i)$ We refer the reader to \cite[Theorem 7.2 and Remark 11.12 $(iii)$]{K85}. See also \cite[Lemma 3.4]{LLRV25}.

    $(ii)$ Due to $(i)$, we see that $H_{p,\theta+kp}^{n+k,(n+k)}$ is the same as $\widetilde{W}_{p,x_d^{\theta-d+kp}}^{n+k}(\bR^d_+)$ introduced in \cite[Definition 2.9.2/2]{T78}, whose definition already includes the density of smooth functions. Thus, the claim follows from \cite[Theorem 2.9.2/2]{T78}. The lemma is proved.
\end{proof}

In the following lemma, we derive some consequences of Hardy's inequality.

\begin{lemma}
Suppose that $p\in(1,\infty)$, $n\in\bN$, and $\theta\in(d-1,\infty)$. 

(i) Assume further that $\theta\in(d-1,d-1+p)$. Let $u\in H_{p,\theta}^{n,(n)}$ satisfy $T_nu=0$ in \eqref{eq5061624}. Then $u$ can be approximated by functions in $C_c^\infty(\bR^d_+)$, and
\begin{align} \label{eq5061659}
    \sum_{r=0}^{n}\|x_d^{r-n}D_x^ru\|_{L_{p,\theta}} \leq N\|u\|_{H_{p,\theta}^{n,(n)}},
\end{align}
where $N=N(d,p,\theta,n)$.
In particular, $u\in H_{p,\theta-np}^{n}$.

(ii) For any $k\in\bN_0$, we have
\begin{equation} \label{eq5071421}
    H_{p,\theta+kp}^{n+k,(n+k)} \subset H_{p,\theta}^{n+k,(n)}.
\end{equation}
\end{lemma}

\begin{proof}
    $(i)$ 
    The first claim follows from \cite[Theorem 2.9.2/2(b)]{T78}. Thus, we may assume that $u\in C_c^\infty(\bR^d_+)$. Let us consider $D_{d}^{r_d}D_{x'}^{r'}u$, where $r=r_d+r'\leq n-1$.
    Due to one-dimensional Hardy's inequality (see e.g. Theorem 5.2 in the preface of \cite{K85}), for each $x'\in \bR^{d-1}$,
    \begin{align*}
        &\int_{0}^{\infty} |(D_{d}^{r_d}D_{x'}^{r'}u)(x',x_d)|^p x_d^{\theta-d+(r-n)p} dx_d
        \\
        &\leq N \int_{0}^{\infty} |(D_{d}^{r_d+1}D_{x'}^{r'}u)(x',x_d)|^p x_d^{\theta-d+(r-n+1)p} dx_d.
    \end{align*}
    Here, we note that $(D_{d}^{r_d}D_{x'}^{r'}u)(x',0)=0$ and $\theta-d+(r-n+1)p\leq \theta-d<p-1$, which allow us to use Hardy's inequality.
    Iterating this inequality and integrating both sides with respect to $x'$, we get
    \begin{align*}
        \|x_d^{r-n}D_x^ru\|_{L_{p,\theta}} \leq N \|D_x^nu\|_{L_{p,\theta}} \leq  N\|u\|_{H_{p,\theta}^{n,(n)}},
    \end{align*}
    which yields \eqref{eq5061659}.

    $(ii)$
By Lemma \ref{lem5061739} $(i)$ we may assume that $u\in C_c^\infty(\overline{\bR^d_+})\subset H_{p,\theta+kp}^{n+k,(n+k)}$. Since $\theta-d>-1$ and $\lim_{x_d\to\infty}u(x',x_d)=0$, we can use Hardy's inequality repeatedly to obtain that for $r_1\leq k$ and $r_2\leq n$,
\begin{align*}
    \|x_d^{r_1}D_x^{r_1+r_2}u\|_{L_{p,\theta}} \leq N \|x_d^kD_x^{r_2+k}u\|_{L_{p,\theta}} \leq N\|u\|_{H_{p,\theta+kp}^{n+k,(n+k)}}.
\end{align*}
    This obviously implies \eqref{eq5071421}.
    The lemma is proved.
\end{proof}

Now we are ready to consider $H_{p,\theta}^{\gamma,(n)}$ for general $\gamma\geq n$.

\begin{lemma} \label{lem5071526}
    Suppose that $p\in(1,\infty)$, $n\in\bN$, $\gamma\geq n$, and $\theta\in(d-1,d-1+p)$. Then $C_c^\infty(\overline{\bR^d_+})$ is dense in $H_{p,\theta}^{\gamma,(n)}$.
\end{lemma}

\begin{proof}
    Let $k\in\bN_0$ be such that $k\leq \gamma-n<k+1$. By Lemma \ref{lem5061739} $(ii)$, we can consider the extension operator
    \begin{equation} \label{eq5071615}
        S_{n,k+1}: \prod_{s=0}^{n-1}W_{p}^{n-\frac{\theta-d+1}{p}-s}(\bR^{d-1}) \to H_{p,\theta +(k+1)p}^{n+k+1,(n+k+1)}.
    \end{equation}
    Note that $H_{p,\theta}^{\gamma,(n)} \subset H_{p,\theta}^{n,(n)}$, which shows that $T_n$ is a well defined operator from $H_{p,\theta}^{\gamma,(n)}$ to the domain of $S_{n,k+1}$.
    Let $u\in H_{p,\theta}^{\gamma,(n)}$, and define
    \begin{equation*}
        v:=S_{n,k+1}T_nu \in H_{p,\theta+(k+1)p}^{n+k+1,(n+k+1)}.
    \end{equation*}
     Then obviously, $T_nv=T_nu$.
    To prove the desired result, we show that $w:=u-v$ and $v$ can each be approximated by functions in $C_c^\infty(\overline{\bR^d_+})$. Here, we note that $u$ and $v$ are in the same space $H_{p,\theta}^{\gamma,(n)}$ due to \eqref{eq5071421}, i.e., the subtraction $w=u-v$ is well defined.

Let us deal with $w$. Take $\eta(x_d)\in C^\infty(\bR)$ such that $\eta=1$ for $|x_d|\geq2$, and $\eta=0$ for $|x_d|\leq 1$. Denote $\eta_l(x_d):= \eta(l^{-1}\log x_d)$. We claim that for $r\leq n$,
        \begin{align*}
        \|D_x^{r}((1-\eta_l) w)-D_x^{r}w\|_{H_{p,\theta}^{\gamma-n}} = \|D_x^{r}(\eta_l w)\|_{H_{p,\theta}^{\gamma-n}} \to0 \text { as } l\to \infty.
    \end{align*}
    To show this, it suffices to prove that for $r_1+r_2\leq n$,
        \begin{align*}
        \lim_{l\to\infty} \|D_d^{r_1}\eta_l D_x^{r_2}w\|_{H_{p,\theta}^{\gamma-n}} = 0.
    \end{align*}
        Let $\zeta \in C_c^\infty(\bR_+)$ satisfying \eqref{eq5062057}. Then on the support of $\zeta$, every derivative of $e^{mr_1}(D_d^{r_1}\eta_l)(e^mx_d)$ is uniformly bounded with respect to $m$ and $l$. By a standard multiplier inequality (see e.g. \cite[Corollary 4.2.2]{Triebel2}),
        \begin{align*}
            \|(D_x^{r_1}\eta_l)(e^m\cdot) (D_x^{r_2}w)(e^m\cdot) \zeta\|_{H_p^{\gamma-n}} \leq Ne^{-mr_1}\|(D_x^{r_2}w)(e^m\cdot) \zeta\|_{H_p^{\gamma-n}}.
        \end{align*}
        Notice that for each $m\in \bZ$, $(D_x^{r_1}\eta_l)(e^m\cdot) \zeta$ vanishes when $l$ is sufficiently large. Therefore, by the dominated convergence theorem, \eqref{defH}, and \eqref{eq5062159}, it suffices to show that
        \begin{equation} \label{eq5062216}
            \sum_{m\in\bZ} e^{m(\theta-r_1p)}\|(D_x^{r_2}w)(e^m\cdot) \zeta\|_{H_p^{\gamma-n}}^p \leq N \|x_d^{-r_1}D_x^{r_2}w\|_{H_{p,\theta}^{\gamma-n}}^p <\infty.
        \end{equation}
        By applying \eqref{eq5062157}, for $k=0,1,\dots,r_1-1$,
        \begin{equation*}
            \|x_d^{-r_1+k}D_x^{r_2+k}w\|_{H_{p,\theta}^{\gamma-n}} \leq N \left( \|x_d^{-r_1+k}D_x^{r_2+k}w\|_{L_{p,\theta}} + \|x_d^{-r_1+k+1}D_x^{r_2+k+1}w\|_{H_{p,\theta}^{\gamma-n}} \right).
        \end{equation*}
        Using this, one can deduce that
        \begin{align} \label{eq5102343}
            \|x_d^{-r_1}D_x^{r_2}w\|_{H_{p,\theta}^{\gamma-n}} &\leq N \sum_{k=0}^{r_1} \|x_d^{-r_1+k} D_x^{r_2+k}w\|_{L_{p,\theta}} + N\|D_x^{r_1+r_2}w\|_{H_{p,\theta}^{\gamma-n}}.
        \end{align}
        Since $w\in H_{p,\theta}^{\gamma,(n)} \subset H_{p,\theta}^{n,(n)}$ and $T_nw=0$, by \eqref{eq5061659}, for $0\leq k\leq r_1-1$,
        \begin{align*}
            \|x_d^{-r_1+k} D_x^{r_2+k}w\|_{L_{p,\theta}} \leq N \|w\|_{H_{p,\theta}^{r_1+r_2,(r_1+r_2)}} \leq N \|w\|_{H_{p,\theta}^{n,(n)}}.
        \end{align*}
        This together with \eqref{eq5102343} yields that
        \begin{align*}
            \|x_d^{-r_1}D_x^{r_2}w\|_{H_{p,\theta}^{\gamma-n}} &\leq N \|w\|_{H_{p,\theta}^{\gamma,(n)}}<\infty,
        \end{align*}
        which proves \eqref{eq5062216}. Therefore, $w$ can be approximated by functions whose supports are contained in strips $a<x_d<b$, with $(a,b)$ varying from function to function.

        When $w$ is supported in a strip, say $x_d\in (a,b)$, then it can be regarded as an element of $\cD'(\bR^d)$, and we have
        \begin{align*}
            \|w\|_{H_{p,\theta}^{\gamma,(n)}} = \sum_{r=0}^{n} \|D_x^rw\|_{H_{p,\theta}^{\gamma-n}} \approx \sum_{r=0}^{n} \|D_x^rw\|_{H_{p}^{\gamma-n}(\bR^d)} \approx \|w\|_{H_p^{\gamma}(\bR^d)}.
        \end{align*}
        See also \cite[Remark 1.12]{Khalf}. This implies that $w$ can be approximated by functions in $C_c^\infty(\bR^d_+)$. This proves the assertion for $w$.

        For $v$, by Lemma \ref{lem5061739} $(i)$, $v$ can be approximated by functions in $C_c^\infty(\overline{\bR^d_+})$ in the space $H_{p,\theta+(k+1)p}^{n+k+1,(n+k+1)}$. Thus, \eqref{eq5071421} leads to the desired result.
The lemma is proved.        
\end{proof}

Below we state the trace theorem for $H_{p,\theta}^{\gamma,(n)}$ when $\gamma\geq n$.

\begin{lemma}
    Suppose that $p\in(1,\infty)$, $n\in\bN$, $\gamma\geq n$, and $\theta\in(d-1,d-1+p)$. The operator $T_n$ extends to a bounded linear operator from $H_{p,\theta}^{\gamma,(n)}$ onto
    \begin{equation} \label{eq5061616}
        \prod_{s=0}^{n-1}W_{p}^{n-\frac{\theta-d+1}{p}-s}(\bR^{d-1}).
    \end{equation}
        Moreover, there is an extension operator $S_{n,\gamma}$ from \eqref{eq5061616} to $H_{p,\theta}^{\gamma,(n)}$ such that $T_nS_{n,\gamma} =I$.
\end{lemma}

\begin{proof}
Note that $H_{p,\theta}^{\gamma,(n)} \subset H_{p,\theta}^{n,(n)}$ and $T_n$ is well defined on $H_{p,\theta}^{n,(n)}$. Therefore, $T_n$ is also well defined on $H_{p,\theta}^{\gamma,(n)}$. In particular, by Lemma \ref{lem5071526}, the operator is obtained as the extension of \eqref{eq5251711} on $C_c^\infty(\overline{\bR^d_+})$.

Let $k\in\bN_0$ be such that $k\leq \gamma-n<k+1$. For $h$ in \eqref{eq5061616} and $S_{n,\gamma}:=S_{n,k+1}$ defined in \eqref{eq5071615}, using \eqref{eq5071421},
\begin{align*}
    \|S_{n,\gamma} h\|_{H_{p,\theta}^{\gamma,(n)}} \leq \|S_{n,k+1} h\|_{H_{p,\theta}^{n+k+1,(n)}} \leq N \|S_{n,k+1} h\|_{H_{p,\theta +(k+1)p}^{n+k+1,(n+k+1)}}.
\end{align*}
This implies that $S_{n,\gamma}$ is bounded from \eqref{eq5061616} to $H_{p,\theta}^{\gamma,(n)}$. Finally, $T_nS_{n,\gamma}=I$ follows from $S_{n,\gamma}=S_{n,k+1}$ and Lemma \ref{lem5061739} $(ii)$.
The lemma is proved.
\end{proof}

We now turn to the interpolation result.

\begin{lemma} \label{lem5090008}
    Let $p\in(1,\infty)$, $n\in\bN$, $\kappa\in(0,1)$, $\gamma_i\geq n$, $i=1,2$, $\theta\in(d-1,d-1+p)$, and $\gamma=\kappa\gamma_1 + (1-\kappa)\gamma_2$.
Then we have
\begin{align*}
H_{p,\theta}^{\gamma,(n)}=[H_{p,\theta}^{\gamma_1,(n)},H_{p,\theta}^{\gamma_2,(n)}]_\kappa,
\end{align*}
where $[A,B]_\kappa$ is the complex interpolation space of $A$ and $B$.
\end{lemma}

\begin{proof}
    First, we prove $[H_{p,\theta}^{\gamma_1,(n)},H_{p,\theta}^{\gamma_2,(n)}]_\kappa \subset H_{p,\theta}^{\gamma,(n)}$. Let $Lu:=(u,D_xu,\cdots, D_x^nu)$ be a bounded operator such that
    \begin{equation*}
        L:H_{p,\theta}^{\gamma_i,(n)} \to \prod_{|\alpha|\leq n} H_{p,\theta}^{\gamma_i-n}, \quad i=1,2.
    \end{equation*}
    By the complex interpolation of operators (see e.g. \cite[Theorem C.2.6]{HNVW16}) and \eqref{eq5071804}, $L$ is a bounded operator from $[H_{p,\theta}^{\gamma_1,(n)},H_{p,\theta}^{\gamma_2,(n)}]_\kappa$ to
    \begin{equation*}
\prod_{|\alpha|\leq n} [H_{p,\theta}^{\gamma_1-n},H_{p,\theta}^{\gamma_2-n}]_\kappa = \prod_{|\alpha|\leq n} H_{p,\theta}^{\gamma-n},
    \end{equation*}
    which implies that
    \begin{equation*}
        \|u\|_{H_{p,\theta}^{\gamma,(n)}} = \sum_{r=0}^{n} \|D_x^ru\|_{H_{p,\theta}^{\gamma-n}} \leq N \|u\|_{[H_{p,\theta}^{\gamma_1,(n)},H_{p,\theta}^{\gamma_2,(n)}]_\kappa}.
    \end{equation*}
    Thus, the first part is proved.

    Next, we deal with $H_{p,\theta}^{\gamma,(n)} \subset [H_{p,\theta}^{\gamma_1,(n)},H_{p,\theta}^{\gamma_2,(n)}]_\kappa$.
    Let $u\in H_{p,\theta}^{\gamma,(n)}$.
     Take $k\in \bN$ satisfying $n+k\geq \max\{\gamma_1,\gamma_2\}$, and consider the trace and extension operators
    \begin{equation*}
        T_n: H_{p,\theta}^{\gamma_i,(n)} \to \prod_{s=0}^{n-1}W_{p}^{n-\frac{\theta-d+1}{p}-s}(\bR^{d-1})
    \end{equation*}
    and
    \begin{equation*}
        S_n = S_{n,k}: \prod_{s=0}^{n-1}W_{p}^{n-\frac{\theta-d+1}{p}-s}(\bR^{d-1}) \to H_{p,\theta+kp}^{n+k,(n+k)},
    \end{equation*}
    which is well defined due to Lemma \ref{lem5061739}. Here, we note that $S_n$ is independent of $i=1,2$. 
    For $u\in H_{p,\theta}^{\gamma_i,(n)}$, by \eqref{eq5071421}, $S_nT_nu \in H_{p,\theta}^{\gamma_i,(n)}$ and $T_n(u-S_nT_nu)=0$ in \eqref{eq5061616}.
    By \eqref{eq5062157} and \eqref{eq5061659}, for any $w\in H_{p,\theta}^{\gamma,(n)}$ such that $T_nw=0$,
    \begin{align} \label{eq9050029}
        \|w\|_{H_{p,\theta-np}^{\gamma}} \leq N \sum_{m=0}^{n-1} \|x_d^{m-n}D_x^mw\|_{L_{p,\theta}} + N \|D_x^nw\|_{H_{p,\theta}^{\gamma-n}} \leq N\|w\|_{H_{p,\theta}^{\gamma,(n)}}.
    \end{align}
    Here, for the last inequality, we used $H_{p,\theta}^{\gamma,(n)} \subset H_{p,\theta}^{n,(n)}$.
    Thus, one can see that $u-S_nT_nu\in H_{p,\theta-np}^{\gamma_i}$. 
    Thus,
    \begin{equation*}
        \cT:H_{p,\theta}^{\gamma_i,(n)} \to \left(H_{p,\theta-np}^{\gamma_i} \cap H_{p,\theta}^{\gamma_i-n}\right) \times H_{p,\theta+kp}^{n+k,(n+k)}, \quad \cT u:=(u-S_nT_nu,S_nT_nu)
    \end{equation*}
    is a bounded linear operator. Let
    \begin{equation*}
        \cS: \left(H_{p,\theta-np}^{\gamma_i} \cap H_{p,\theta}^{\gamma_i-n}\right) \times H_{p,\theta+kp}^{n+k,(n+k)} \to H_{p,\theta}^{\gamma_i,(n)}, \quad \cS ((v_1,v_2))=v_1+v_2.
    \end{equation*}
    Now we verify that $\cS$ is well defined.
    For $m=1,\dots,n-1$, by \eqref{eq5071804} with $(\theta-np,\theta, \gamma_i,\gamma_i-n,m/n)$ in place of $(\theta_1,\theta_2,\gamma_1,\gamma_2,\kappa)$,
    \begin{equation*}
        H_{p,\theta-mp}^{\gamma_i-n+m} = [H_{p,\theta-np}^{\gamma_i}, H_{p,\theta}^{\gamma_i-n}]_{m/n}.
    \end{equation*}
    This and \eqref{equiv_der} imply that for $m=0,\dots,n$,
    \begin{align*}
        \|D_x^mu\|_{H_{p,\theta}^{\gamma_i-n}} &\leq N \|x_d^mD_x^mu\|_{H_{p,\theta-mp}^{\gamma_i-n}} 
        \\
        &\leq N \|u\|_{H_{p,\theta-mp}^{\gamma_i-n+m}} \leq N\left( \|u\|_{H_{p,\theta-np}^{\gamma_i}} + \|u\|_{H_{p,\theta}^{\gamma_i-n}} \right).
    \end{align*}
This together with \eqref{eq5071421} yields that $\cS$ is a well defined bounded linear operator.

Note that $\cS \cT u=u$. Moreover, by \eqref{eq9031054},
\begin{align*}
    Z &:= \left[\left(H_{p,\theta-np}^{\gamma_1} \cap H_{p,\theta}^{\gamma_1-n}\right) \times H_{p,\theta+kp}^{n+k,(n+k)}, \left(H_{p,\theta-np}^{\gamma_2} \cap H_{p,\theta}^{\gamma_2-n}\right) \times H_{p,\theta+kp}^{n+k,(n+k)}\right]_{\kappa}
    \\
    &= \left( H_{p,\theta-np}^{\gamma} \cap H_{p,\theta}^{\gamma-n}\right) \times H_{p,\theta+kp}^{n+k,(n+k)}.
\end{align*}
    Thus, using the complex interpolation of operators (see e.g. \cite[Theorem C.2.6]{HNVW16}),
    \begin{align} \label{eq5090005}
        \|u\|_{[H_{p,\theta}^{\gamma_1,(n)},H_{p,\theta}^{\gamma_2,(n)}]_\kappa} &\leq N \|\cT u\|_{Z} \nonumber
        \\
        &= N \|u-S_nT_nu\|_{H_{p,\theta-np}^{\gamma}} + N \|u-S_nT_nu\|_{H_{p,\theta}^{\gamma-n}} \nonumber
        \\
        &\quad+ N \|S_nT_nu\|_{H_{p,\theta+kp}^{n+k,(n+k)}}.
    \end{align}
    Applying \eqref{eq9050029} to $u-S_nT_nu$, and using \eqref{eq5071421}, we get
    \begin{align*}
        &\|u-S_nT_nu\|_{H_{p,\theta-np}^{\gamma}} + \|u-S_nT_nu\|_{H_{p,\theta}^{\gamma-n}} + \|S_nT_nu\|_{H_{p,\theta+kp}^{n+k,(n+k)}}
        \\
        &\leq N \left( \|u-S_nT_nu\|_{H_{p,\theta}^{\gamma,(n)}} + \|S_nT_nu\|_{H_{p,\theta+kp}^{n+k,(n+k)}} \right) 
        \\
        &\leq N \left( \|u\|_{H_{p,\theta}^{\gamma,(n)}} + \|S_nT_nu\|_{H_{p,\theta+kp}^{n+k,(n+k)}} \right) \leq N \|u\|_{H_{p,\theta}^{\gamma,(n)}}.
    \end{align*}
    For the last inequality, we used that both $T_n$ and $S_n$ are bounded operators.
    This and \eqref{eq5090005} imply that $u\in [H_{p,\theta}^{\gamma_1,(n)},H_{p,\theta}^{\gamma_2,(n)}]_\kappa$. The lemma is proved.
\end{proof}

We next extend the density and interpolation results to arbitrary $\gamma\in\bR$. To this end, we show that if $c>0$ is sufficiently large, then $\cL_{b,c}$ in \eqref{eq5090021} is an isomorphism between $H_{p,\theta}^{\gamma+2,(n)}$ and $H_{p,\theta}^{\gamma,(n)}$.

\begin{lemma} \label{lem4161837}
    Let $p\in(1,\infty)$,  $\gamma\in\bR$, and $n\in\bN$. 
    
    (i) Let $\theta\in(d-1,d-1+p)$. Then the space $C_c^\infty(\overline{\bR^d_+})$ is dense in $H_{p,\theta}^{\gamma,(n)}$. Moreover, the set of all $l_2$-valued functions $u=(u^1,u^2,\cdots)$ such that $u^k=0$ for all sufficiently large $k$, and each component $u^k \in C_c^\infty(\overline{\bR^d_+})$, is dense in $H_{p,\theta}^{\gamma,(n)}(l_2)$.

    (ii) 
    Let $b,c,\theta\in\bR$. Then there is $c_0=c_0(d,p,\theta,\gamma,b,n)>0$ such that if $c>c_0$, then $\cL_{b,c}$ is bounded and invertible from $H_{p,\theta}^{\gamma+2,(n)}$ to $H_{p,\theta}^{\gamma,(n)}$.
    In addition, the assertion $(ii)$ also holds true for $H_{p,\theta}^{\gamma,(n)}(l_2)$ in place of $H_{p,\theta}^{\gamma,(n)}$.

    (iii) (Interpolation)
Let $\kappa\in(0,1)$, $\theta\in(d-1,d-1+p)$, $\gamma_i\in\bR$, $i=1,2$, and $\gamma=\kappa\gamma_1 + (1-\kappa)\gamma_2$.
Then we have
\begin{align} \label{eq5112341}
H_{p,\theta}^{\gamma,(n)} = [H_{p,\theta}^{\gamma_1,(n)},H_{p,\theta}^{\gamma_2,(n)}]_\kappa,
\end{align}
where $[A,B]_\kappa$ is the complex interpolation space of $A$ and $B$.
\end{lemma}

\begin{proof}
We first prove $(ii)$ and then derive $(i)$ and $(iii)$ from $(ii)$.
    
    $(ii)$ 
    Note that for $r_d,r'\in\bN_0$,
    \begin{align} \label{eq5101318}
       D_d^{r_d}D_{x'}^{r'} \cL_{b,c}u &= \cL_{b+2r_d,c-br_d-r_d(r_d-1)}D_d^{r_d}D_{x'}^{r'}u + 2 1_{r_d\geq1} r_dx_d\Delta_{x'} D_d^{r_d-1}D_{x'}^{r'}u \nonumber
       \\
       &\quad + r_d(r_d-1)1_{r_d\geq2}\Delta_{x'} D_d^{r_d-2}D_{x'}^{r'}u,
    \end{align}
    where $\Delta_{x'} = D_{1}^2 + D_2^2 + \cdots + D_{d-1}^2$.
    By \eqref{equiv_der}, if $u \in H_{p,\theta}^{\gamma+2,(n)}$, then for any $b,c\in\bR$ and $r_d+r'\leq n$, $D_d^{r_d}D_{x'}^{r'} \cL_{b,c}u \in H_{p,\theta}^{\gamma-n}$, which implies that 
    $\cL_{b,c}$ is a bounded operator from $H_{p,\theta}^{\gamma+2,(n)}$ to $H_{p,\theta}^{\gamma,(n)}$.  Thus, we only need to prove that its inverse is bounded.
    By Lemma \ref{lem4282020}, there is $c=c(d,p,\theta,\gamma,b,n)>0$ such that for every $r_d\leq n$, $\cL_{b+2r_d,c-br_d-r_d(r_d-1)}$ is a bounded isomorphism from $H_{p,\theta}^{\gamma+2-n}$ onto $H_{p,\theta}^{\gamma-n}$.
    For $u\in H_{p,\theta}^{\gamma,(n)}$ and $r_d+r'\leq n$, by \eqref{eq5101318},
    \begin{align*}
        &\cL_{b+2r_d,c-br_d-r_d(r_d-1)}(D_d^{r_d}D_{x'}^{r'}\cL_{b,c}^{-1}u) 
        \\
        &= D_d^{r_d}D_{x'}^{r'}u -21_{r_d\geq1} r_d\sum_{i=1}^{d-1}x_dD_i\left(D_d^{r_d-1} \cL^{-1}_{b,c}(D_iD_{x'}^{r'}u)\right)
        \\
        &\quad-r_d(r_d-1)1_{r_d\geq2}D_d^{r_d-2}\cL_{b,c}^{-1}(\Delta_{x'}D_{x'}^{r'}u)
        \\
        &=:I_1+I_2+I_3.
    \end{align*}
    When $r_d=0$, the left-hand side of the equality becomes $\cL_{b,c}(D_{x'}^{r'}\cL_{b,c}^{-1})u = D_{x'}^{r'}u \in H_{p,\theta}^{\gamma-n}$, which easily leads to $D_{x'}^{r'}\cL_{b,c}^{-1}u \in H_{p,\theta}^{\gamma+2-n}$. Here, we used that $\cL_{b,c}^{-1}$ commutes with $D_{x'}$, which follows from the commutativity of $\cL_{b,c}$ and $D_{x'}$.
    Fix $n_0\leq n$ and assume that 
    \begin{equation} \label{eq5110121}
        D_d^{r_d}(\cL_{b,c}^{-1}D_{x'}^{r'}u) \in H_{p,\theta}^{\gamma+2-n}
    \end{equation}
    for $r_d\leq n_0-1$ and $r'\in\bN_0$ such that $r_d+r'\leq n$. It is obvious that $I_1 \in H_{p,\theta}^{\gamma-n}$.
    For $I_2$, using $u\in H_{p,\theta}^{\gamma,(n)}$, \eqref{equiv_der}, and \eqref{eq5110121} with $r'+1$ instead of $r'$, we get $I_2\in H_{p,\theta}^{\gamma+1-n}\subset H_{p,\theta}^{\gamma-n}$. Similarly, one can also prove that $I_3\in H_{p,\theta}^{\gamma-n}$.
    Since $\cL_{b+2r_d,c-br_d-r_d(r_d-1)}$ is a bounded isomorphism from $H_{p,\theta}^{\gamma+2-n}$ onto $H_{p,\theta}^{\gamma-n}$, we still have \eqref{eq5110121} when $r_d=n_0$. By mathematical induction in $r_d$, we see that $D_x^r\cL_{b,c}^{-1}u \in H_{p,\theta}^{\gamma+2-n}$ for $r\leq n$. The claim is obtained.

        $(i)$ The case $\gamma\geq n$ is proved in Lemma \ref{lem5071526}. Let $\gamma<n$ and find $m\in\bN$ so that $2m+\gamma\geq n$. Due to $(ii)$, we can choose $b,c\in\bR$ so that $\cL_{b,c}^{-m}u \in H_{p,\theta}^{2m+\gamma,(n)}$. Using the case $\gamma\geq n$, there is $v_l \in C_c^\infty(\overline{\bR^d_+})$ so that $v_l\to \cL_{b,c}^{-m}u$ in $H_{p,\theta}^{2m+\gamma,(n)}$, which is equivalent to $\cL_{b,c}^m v_l \to u$ in $H_{p,\theta}^{\gamma,(n)}$. Here, we remark that $\cL_{b,c}^m v_l \in C_c^\infty(\overline{\bR^d_+})$.

        $(iii)$
        Note that the relation $[H_{p,\theta}^{\gamma_1,(n)},H_{p,\theta}^{\gamma_2,(n)}]_\kappa \subset H_{p,\theta}^{\gamma,(n)}$ can be proved by using the fact that the first part of the proof of Lemma \ref{lem5090008} remains valid for arbitrary $\gamma_1$ and $\gamma_2$. 

        Let us consider the reverse inclusion $H_{p,\theta}^{\gamma,(n)} \subset [H_{p,\theta}^{\gamma_1,(n)},H_{p,\theta}^{\gamma_2,(n)}]_\kappa$.
        When both $\gamma_1,\gamma_2\geq n$, the desired result follows from Lemma \ref{lem5090008}. For general $\gamma_1$ and $\gamma_2$, take $m\in \bN$ so that $\gamma_1+2m,\gamma_2+2m\geq n$. Then by $(ii)$ and the complex interpolation of operators, one can choose $b,c\in\bR$ so that
        \begin{equation*}
            \cL_{b,c}^m:H_{p,\theta}^{\gamma+2m,(n)} \to [H_{p,\theta}^{\gamma_1,(n)},H_{p,\theta}^{\gamma_2,(n)}]_\kappa.
        \end{equation*}
        This actually implies the desired inclusion.
        The lemma is proved.
\end{proof}

\subsection{Comparison with Classical Weighted Sobolev Spaces}

In this subsection, we compare $H_{p,\theta}^{\gamma,(n)}$ with classical weighted spaces near the boundary $x_d=0$. We first introduce the classical weighted Besov and Triebel–Lizorkin spaces on the whole space $\bR^d$. Below we adopt the notation used in \eqref{eq4070111}.
For $p,q>1$, $\theta\in\bR$, and $\gamma\in\bR$, the Besov space $B_{p,q}^\gamma(\bR^d,|x_d|^{\theta-d})$ is the set of all tempered distributions $u\in \cS'(\bR^d)$ such that
\begin{equation*}
    \|u\|_{B_{p,q}^\gamma(\bR^d,|x_d|^{\theta-d})}:=\|S_0u\|_{L_p(|x_d|^{\theta-d}dx)} + \left( \sum_{j=1}^\infty 2^{j q \gamma} \|\Delta_j u\|_{L_p(|x_d|^{\theta-d}dx)}^q \right)^{1/q} <\infty,
\end{equation*}
where
\begin{equation*}
    \|v\|_{L_p(|x_d|^{\theta-d}dx)}^p := \int_{\bR^d} |v|^p |x_d|^{\theta-d} dx.
\end{equation*}
We also introduce the Triebel–Lizorkin space $F_{p,q}^\gamma(\bR^d,|x_d|^{\theta-d})$, which is the set of all $u\in \cS'(\bR^d)$ such that
\begin{equation*}
    \|u\|_{F_{p,q}^\gamma(\bR^d,|x_d|^{\theta-d})}:=\|S_0u\|_{L_p(|x_d|^{\theta-d}dx)} +  \left\|\left(\sum_{j=1}^\infty 2^{j q \gamma} |\Delta_j u|^q \right)^{1/q}\right\|_{L_p(|x_d|^{\theta-d}dx)}  <\infty.
\end{equation*}
It is known in \cite[Theorems 1.10 and 2.20]{R01} that when $n\in\bN_0$ and $\theta\in(d-1,d-1+p)$, $F_{p,2}^{n}(\bR^d,|x_d|^{\theta-d}) = W_{p}^{n}(\bR^d,|x_d|^{\theta-d})$, where
\begin{equation*}
    W_{p}^{n}(\bR^d,|x_d|^{\theta-d}):=\{u\in L_{p}(\bR^d,|x_d|^{\theta-d}dx): u,\cdots,D_x^nu \in L_{p}(\bR^d,|x_d|^{\theta-d}dx) \}.
\end{equation*}

We also review the definitions of the Besov and Triebel–Lizorkin spaces on domains. The space $B_{p,q}^{\gamma}(\bR^d_+,x_d^{\theta-d})$ is the collection of all restrictions of elements of $B_{p,q}^{\gamma}(\bR^d,|x_d|^{\theta-d})$ on $\bR^d_+$. The norm in this space is defined as
\begin{equation*}
    \|u\|_{B_{p,q}^{\gamma}(\bR^d_+,x_d^{\theta-d})} = \inf \|v\|_{B_{p,q}^{\gamma}(\bR^d,|x_d|^{\theta-d})},
\end{equation*}
where the infimum is taken over all $v\in B_{p,q}^{\gamma}(\bR^d,|x_d|^{\theta-d})$ whose restriction on $\bR^d_+$ is $u$. The space $F_{p,q}^{\gamma}(\bR^d_+,x_d^{\theta-d})$ can be defined in a similar way.

In the special case when $\gamma=n\in \bN_0$, we have the following.

\begin{lemma} \label{lem5141224}
    Let $p>1$, $n\in\bN_0$, and $\theta\in(d-1,d-1+p)$. Then we have $F_{p,2}^{n}(\bR^d_+,x_d^{\theta-d})=H_{p,\theta}^{n,(n)}$.
\end{lemma}

\begin{proof}
Since the case $n=0$ is obvious, it suffices to handle the case $n\geq1$.
We first prove $F_{p,2}^{n}(\bR^d_+,x_d^{\theta-d}) \subset H_{p,\theta}^{n,(n)}$. Suppose that $u\in F_{p,2}^{n}(\bR^d_+,x_d^{\theta-d})$. Then there exists $v\in F_{p,2}^{n}(\bR^d,|x_d|^{\theta-d})$ such that $v=u$ on $\bR^d_+$. Due to the relation $F_{p,2}^{n}(\bR^d,|x_d|^{\theta-d}) = W_{p}^{n}(\bR^d,|x_d|^{\theta-d})$, one can easily show that $u\in H_{p,\theta}^{n,(n)}$.

It remains to show the reverse inclusion. To prove this, it suffices to show that for $u\in H_{p,\theta}^{n,(n)}$, there is $v\in F_{p,2}^{n}(\bR^d,|x_d|^{\theta-d})$ such that $v=u$ on $\bR^d_+$. Due to Lemma \ref{lem5071526}, we may assume that $u\in C_c^\infty(\overline{\bR^d_+})$. By \cite[Theorem 5.19]{AF03}, $u$ can be extended to the whole space without losing its weak differentiability up to any given order. More precisely, choose constants $c_1,\ldots,c_n$ such that
\begin{equation*}
\sum_{j=1}^{n} c_j(-j)^k=1,\qquad k=0,\ldots,n-1.
\end{equation*}
Define $v=u$ on $\bR^d_+$ and
\begin{equation*}
v(x',x_d) := \sum_{j=1}^{n}c_j u(x',-jx_d), \qquad x_d<0.
\end{equation*}
Then we have
\begin{equation*}
    \|v\|_{W_p^{n}(\bR^d,|x_d|^{\theta-d})} \leq N \|u\|_{W_{p}^{n}(\bR^d_+,x_d^{\theta-d})}.
\end{equation*}
Since $F_{p,2}^{n}(\bR^d,|x_d|^{\theta-d}) = W_{p}^{n}(\bR^d,|x_d|^{\theta-d})$, we get the desired result. The lemma is proved.
\end{proof}

For general $H_{p,\theta}^{\gamma,(n)}$ with $\gamma\in \bR$, we are concerned only with the behavior near the boundary. In other words, we study $I_Ru := u\eta_R$, where $\eta_R\in C^\infty(\bR^d_+)$ is a one-dimensional cutoff function such that $\eta_R=1$ on $\{0<x_d<R\}$ and $\eta_R=0$ on $\{x_d>2R\}$.

\begin{lemma}
    Let $R>0$, $p\in(1,\infty)$, $n\in\bN_0$, and $\theta\in (d-1,d-1+p)$.

    (i) If $\gamma\geq n$, then
    \begin{equation*}
        \|I_Ru\|_{H_{p,\theta}^{\gamma,(n)}} \leq N \|u\|_{F_{p,2}^{\gamma}(\bR^d_+,x_d^{\theta-d})},
    \end{equation*}
    where $N=N(R)$.

    (ii) If $\gamma\leq n$, then
        \begin{equation} \label{eq5121127}
        \|I_Ru\|_{F_{p,2}^{\gamma}(\bR^d_+,x_d^{\theta-d})} \leq N \|u\|_{H_{p,\theta}^{\gamma,(n)}},
    \end{equation}
    where $N=N(R)$.
\end{lemma}

\begin{proof}
    $(i)$ When $\gamma\in \bN_0$ such that $\gamma\geq n$,
    \begin{align*}
        \|I_R u\|_{H_{p,\theta}^{\gamma-n}} \leq N \sum_{r=0}^{\gamma-n} \|x_d^{r} D_x^r(I_Ru)\|_{L_{p,\theta}} \leq N(R) \sum_{r=0}^{\gamma-n} \|D_x^r(I_Ru)\|_{L_{p,\theta}}.
    \end{align*}
    Thus, one can deduce that
    \begin{equation*}
        \|I_R u\|_{H_{p,\theta}^{\gamma,(n)}} \leq N \|I_R u\|_{W_{p}^{\gamma}(\bR^d_+,x_d^{\theta-d})} \leq N \|u\|_{W_{p}^{\gamma}(\bR^d_+,x_d^{\theta-d})}= N\|u\|_{F_{p,2}^{\gamma}(\bR^d_+,x_d^{\theta-d})}.
    \end{equation*}
    This implies that $I_R:F_{p,2}^{\gamma}(\bR^d_+,x_d^{\theta-d})\to H_{p,\theta}^{\gamma,(n)}$ is a bounded linear operator if $\gamma\in\bN_0$ and $\gamma\geq n$. 
    By \cite[Proposition 5.6]{LMV18},
\begin{align} \label{eq5132348}
F_{p,2}^{\gamma}(\bR^d_+,x_d^{\theta-d}) = [F_{p,2}^{\gamma_1}(\bR^d_+,x_d^{\theta-d}),F_{p,2}^{\gamma_2}(\bR^d_+,x_d^{\theta-d})]_\kappa,
\end{align}
where $\kappa\in(0,1)$, $\gamma_i\in\bR$, $i=1,2$, and $\gamma=\kappa\gamma_1 + (1-\kappa)\gamma_2$.
  Thus, the general case $\gamma\geq n$ follows from \eqref{eq5112341}, \eqref{eq5132348}, and the complex interpolation of operators.

    $(ii)$ First, we assume that $n=0$. In this case, $H_{p,\theta}^{\gamma,(n)} = H_{p,\theta}^{\gamma}$, and by Lemma \ref{lem4161836} $(vi)$, the dual space of $H_{p,\theta}^{\gamma}$ is $H_{p',\theta'}^{-\gamma}$, where $1/p+1/p'=1$ and $\theta/p+\theta'/p'=d$.
Since $-\gamma\geq n=0$, $I_R:F_{p',2}^{-\gamma}(\bR^d_+,x_d^{\theta'-d}) \to H_{p',\theta'}^{-\gamma}$ is bounded. This implies that its dual operator $I_R^*$ is bounded from $H_{p,\theta}^{\gamma}$ to $(F_{p',2}^{-\gamma}(\bR^d_+,x_d^{\theta'-d}))^*$, which is the dual space of $F_{p',2}^{-\gamma}(\bR^d_+,x_d^{\theta'-d})$.
Here, one can deduce that $I_R^*=I_R$.

Let $u\in H_{p,\theta}^{\gamma}$ and let $\cR:F_{p',2}^{-\gamma}(\bR^d,|x_d|^{\theta'-d}) \to F_{p',2}^{-\gamma}(\bR^d_+,x_d^{\theta'-d})$ be the standard restriction operator. 
Define a functional $\tilde{I}_{R,u}$ on $F_{p',2}^{-\gamma}(\bR^d,|x_d|^{\theta'-d})$ by
    \begin{equation*}
        \tilde{I}_{R,u}(v) := (I_R(u))(\cR v) = u(I_R\cR v), \quad v\in F_{p',2}^{-\gamma}(\bR^d,|x_d|^{\theta'-d}).
    \end{equation*}
    Actually,
    \begin{equation*}
        |\tilde{I}_{R,u}(v)| \leq N(R) \|u\|_{H_{p,\theta}^{\gamma}} \|\cR v\|_{F_{p',2}^{-\gamma}(\bR^d_+,x_d^{\theta'-d})} \leq N(R) \|u\|_{H_{p,\theta}^{\gamma}} \|v\|_{F_{p',2}^{-\gamma}(\bR^d,|x_d|^{\theta'-d})},
    \end{equation*}
    which means that $\tilde{I}_{R,u}$ is in $F_{p,2}^{\gamma}(\bR^d,|x_d|^{\theta-d})$, the dual space of $F_{p',2}^{-\gamma}(\bR^d,|x_d|^{\theta'-d})$ (see \cite[Proposition 3.5]{MV15}).
    Moreover, if we further assume that $v\in C_c^\infty(\bR^d_+)$, then $\cR v = v$, which yields $\tilde{I}_{R,u}=I_R(u)$ in $\cD'(\bR^d_+)$. Hence, $I_R$ is a bounded operator from $H_{p,\theta}^{\gamma}$ to $F_{p,2}^{\gamma}(\bR^d_+,x_d^{\theta-d})$. The case $n=0$ is proved.

    Next, we consider general $n\geq 1$.
    By investigating the proof of \cite[Proposition 5.6]{LMV18}, there is an extension $v_R \in F_{p,2}^{\gamma-n}(\bR^d,|x_d|^{\theta-d})$ of $I_Ru$ such that, for $r=0,1,\dots,n$,
\begin{equation*}
    \|D_x^r(v_R)\|_{F_{p,2}^{\gamma-n}(\bR^d,|x_d|^{\theta-d})} \leq N \|D_x^r(I_Ru)\|_{F_{p,2}^{\gamma-n}(\bR^d_+,x_d^{\theta-d})}.
\end{equation*}
    Thus, by \cite[Propositions 2.4 and 3.10]{MV12},
    \begin{align*}
        \|I_Ru\|_{F_{p,2}^{\gamma}(\bR^d_+,x_d^{\theta-d})} &\leq N \sum_{r=0}^{n} \|D_x^r(I_Ru)\|_{F_{p,2}^{\gamma-n}(\bR^d_+,x_d^{\theta-d})}
        \\
        &\leq N(R) \sum_{r=0}^{n} \|I_{2R}D_x^ru\|_{F_{p,2}^{\gamma-n}(\bR^d_+,x_d^{\theta-d})}.
    \end{align*}
    Thus, by the case $n=0$, 
\begin{align*}
    \|I_Ru\|_{F_{p,2}^{\gamma}(\bR^d_+,x_d^{\theta-d})} \leq N(R) \sum_{r=0}^{n} \|I_{2R}D_x^ru\|_{H_{p,\theta}^{\gamma-n}} \leq N(R) \|u\|_{H_{p,\theta}^{\gamma,(n)}}.
\end{align*}
The lemma is proved.
\end{proof}

\subsection{Stochastic Banach spaces}

In this subsection, we introduce some basic but essential properties of our solution space $\frH_{neu}^{\gamma+2}(p,\theta,\tau)$

\begin{definition}
    Let $\tau$ be a stopping time.
    
    $(i)$ 
    We denote by $\bH_0^\infty(\tau)$ the set of all functions of the type
\begin{equation} \label{eq4221514}
    f(t,x)=\sum_{i=1}^{n}1_{(\tau_{i-1},\tau_{i}]}(t) f_{i}(x)
\end{equation}
where $f_{i}\in C_c^\infty(\bR^d_+)$ and $0\leq \tau_{0} \leq \cdots \leq \tau_{n} \leq \tau$ are bounded stopping times.

$(ii)$ 
We define
 $\bH_0^\infty(\tau,l_2)$ to be the set of all $l_2$-valued functions $g=(g^1,g^2,\cdots)$ such that $g^k=0$ for all sufficiently large $k$, and each component $g^k$ is in $\bH_0^\infty(\tau)$.

 $(iii)$
We define $\overline{\bH}_0^{\infty}(\tau)$ as the set defined by replacing $f_{i}\in C_c^\infty(\bR^d_+)$ in \eqref{eq4221514} with $f_{i}\in C_c^\infty(\overline{\bR^d_+})$.
Moreover, $\overline{\bH}_0^{\infty}(\tau,l_2)$ is defined analogously.
\end{definition}

\begin{lemma}
    Let $p\in(1,\infty)$, $n\in\bN_0$, $\gamma,\theta\in\bR$, and $\tau$ be a stopping time.

    (i) The sets $\bH_0^\infty(\tau)$ and $\bH_0^\infty(\tau,l_2)$ are dense in $\bH_{p,\theta}^{\gamma}(\tau)$ and $\bH_{p,\theta}^{\gamma}(\tau,l_2)$, respectively.

    (ii) When $\theta\in (d-1,d-1+p)$, the sets $\overline{\bH}_0^{\infty}(\tau)$ and $\overline{\bH}_0^{\infty}(\tau,l_2)$ are dense in ${\bH}_{p,\theta}^{\gamma,(n)}(\tau)$ and ${\bH}_{p,\theta}^{\gamma,(n)}(\tau,l_2)$, respectively.
\end{lemma}

\begin{proof}
    The results follow from \cite[Lemma 1.2.19]{HNVW16}, Lemma \ref{lem4161836} $(i)$, and Lemma \ref{lem4161837} $(i)$. We also remark that $(i)$ was already proved in \cite[Theorem 3.10]{Kry99}.
\end{proof}

\begin{proposition}
Let $p\in[2,\infty)$, $\gamma,\theta\in\bR$, $T\in(0,\infty)$, and $\tau\leq T$ be a bounded stopping time.
    For any $u\in \frH_{neu}^{\gamma+2}(p,\theta,\tau)$, we have $u\in L_p(C([0,\tau];H_{p,\theta}^\gamma))$ and
    \begin{equation} \label{eq01072253}
        \left(\bE \sup_{t\leq \tau} \|u(t,\cdot)\|_{H_{p,\theta}^{\gamma}}^p\right)^{1/p} \leq N \|u\|_{\frH_{neu}^{\gamma+2}(p,\theta,\tau)},
    \end{equation}
    where $N=N(d,p,\theta,\gamma,T)$.
\end{proposition}

\begin{proof}
Let $du=f dt + g^kdw_t^k$ with the initial data $u_0$. Then for $\zeta\in C_c^\infty(\bR_+)$,
    \begin{equation*}
        d(u(t,e^nx)\zeta(x)) = f(t,e^nx)\zeta(x) dt + g^k(t,e^nx) \zeta(x) dw_t^k, \quad t>0, x\in \bR^d,
    \end{equation*}
    with the initial data $u(0,e^nx)\zeta(x)=u_0(e^nx)\zeta(x)$. 
Since $u_0\in U_{p,\theta}^{\gamma+2}\subset \fH_{p,\theta}^{\gamma}$,
we have
\begin{equation*}
    u_0(e^n\cdot)\zeta \in L_p(\Omega,\cF_0;H_{p}^{\gamma}).
\end{equation*}
 Moreover, due to $f\in \bH_{p,\theta}^\gamma(\tau)$ and $g\in {\bH}_{p,\theta}^{\gamma+1,(1)}(\tau,l_2)$, we have
    \begin{eqnarray*}
        f(\cdot,e^n\cdot)\zeta \in L_p((0,\tau],\cP;H_{p}^\gamma), \quad g(\cdot,e^n\cdot)\zeta \in L_p((0,\tau],\cP;H_{p}^{\gamma}(l_2)).
    \end{eqnarray*}
    Thus, by repeating the proof of \cite[Theorem 3.7]{Kry99}, one can obtain
    \begin{align*}
        \bE\sup_{t\leq \tau} \|u(t,e^n\cdot)\zeta \|_{H_{p}^\gamma}^p &\leq N \|u_0(e^n\cdot)\zeta\|_{L_p(\Omega,\cF_0;H_{p}^{\gamma})}^p + N\|f(\cdot,e^n\cdot)\zeta \|_{L_p((0,\tau],\cP;H_p^\gamma(\bR^{d}))}^p 
        \\
        &\quad+ N \|g(\cdot,e^n\cdot)\zeta \|_{L_p((0,\tau],\cP;H_p^\gamma(\bR^{d},l_2))}^p.
    \end{align*}
Due to \eqref{defH}, we get
    \begin{equation*}
        \bE\sup_{t\leq \tau} \|u(t,\cdot)\|_{H_{p,\theta}^\gamma}^p \leq \sum_{n\in\bZ} e^{n\theta}\bE\sup_{t\leq \tau} \|u(t,e^n\cdot)\zeta \|_{H_{p}^\gamma}^p,
    \end{equation*}
     which yields that
    \begin{equation} \label{eq2152119}
        \left( \bE \sup_{t\leq \tau} \|u(t,\cdot)\|_{H_{p,\theta}^{\gamma}}^p \right)^{1/p} \leq N\left(\|u_0\|_{\fH_{p,\theta}^{\gamma}} + \|f\|_{\bH_{p,\theta}^\gamma(\tau)} + \|g\|_{\bH_{p,\theta}^\gamma(\tau,l_2)} \right).
    \end{equation}
    It remains to use \eqref{def_norm}.
    The proposition is proved.
\end{proof}

\begin{lemma}
    Let $p\in[2,\infty)$, $\gamma,\theta\in\bR$, $T\in(0,\infty)$, and $\tau\leq T$ be a bounded stopping time. Then $\frH_{neu}^{\gamma+2}(p,\theta,\tau)$ is a Banach space.
\end{lemma}

\begin{proof}
    Suppose that $u_n$ is a Cauchy sequence in $\frH_{neu}^{\gamma+2}(p,\theta,\tau)$. If we write
    \begin{equation*}
        du_n = f_n dt + g^k_n dw_t^k; \quad u_n(0,\cdot)=u_{n0},
    \end{equation*}
    then $u_n, f_n, g_n$, and $u_{n0}$ each converge to $u, f, g$, and $u_0$ in their corresponding spaces as $n\to\infty$. 
    By \eqref{eq01072253}, for any $\phi\in C_c^\infty(\bR^d_+)$ and $t\leq \tau$,
    \begin{equation*}
        (u_n(t,\cdot),\phi) \to (u(t,\cdot),\phi) \, \text{ as } n\to\infty,
    \end{equation*}
    which implies that
        \begin{equation*}
        du = f dt + g^k dw_t^k; \quad u(0,\cdot)=u_{0}
    \end{equation*}
    in the sense of distributions. The lemma is proved.
\end{proof}

\section{A Simple Equation: Existence and Estimates} \label{sec_simple}

In this section, we consider a simple equation
\begin{equation} \label{eq_simple}
du=  \left(\Delta u + f\right) dt + g^kdw_t^k
\end{equation}
with the initial condition $u(0,\cdot)=0$, in the space $\frH_{neu}^{\gamma+2}(p,\theta,\tau)$.
More precisely, we aim to analyze the existence and low order regularity of solutions in the Neumann framework. We point out that in this section, $f$ and $g$ are independent of $u$.
Formally, the solution can be represented as
\begin{align*}
    u(t,x) &= \int_0^t \int_{\bR^d_+} p_{\cN}(t-s,x,y) f(s,y) dyds \nonumber
    \\
    &\quad+ \sum_{k=1}^\infty \int_0^t \int_{\bR^d_+} p_{\cN}(t-s,x,y) g^k(s,y) dydw_s^k
\end{align*}
where $p_{\cN}(t,x,y)$ is the Green function of the operator $\partial_t-\Delta$ in $\bR^d_+$ subject to the homogeneous Neumann boundary condition on the boundary $x_d=0$. More precisely,
\begin{equation*}
    p_{\cN}(t,x,y) = p(t,x-y) + p(t,x-y^*),
\end{equation*}
where $y^*=(y',-y_d)$ and
\begin{eqnarray*}
    p(t,x) = \frac{1}{(4\pi t)^{d/2}} e^{-\frac{|x|^2}{4t}}.
\end{eqnarray*}

\begin{lemma}
    Let $\gamma_0>-1$, $\gamma_1\in\bR$, and $\sigma>0$.
 Then, for any $t>0$ and $x\in \bR^d_+$, 
    \begin{align} \label{eq3102020}
    \int_{\bR^d_+} \frac{{y_d}^{\gamma_0}}{(\sqrt{t}+y_d)^{\gamma_1}} p(\sigma t,x-y) dy \leq N(\sqrt{t}+x_d)^{\gamma_0-\gamma_1},
    \end{align}
    where $N=N(d,\sigma,\gamma_0,\gamma_1)$. 
\end{lemma}

\begin{proof}
    When $t=1$, \eqref{eq3102020} follows from \cite[Lemma 3.7]{KLS21} by taking $\gamma=\gamma_0$ and $\omega=\gamma_1$. In particular, one can put $\alpha=\beta=0$ in the lemma. The general case $t>0$ can be obtained by a scaling argument. The lemma is proved.
\end{proof}

\begin{lemma} \label{lem4162023}
    Let $p\in(1,\infty)$, $\theta\in(d-1,d-1+p)$, $T\in (0,\infty)$, and $\tau\leq T$ be a bounded stopping time. Assume that $f \in \bH_{0}^{\infty}(\tau)$. Then there is a solution $u$ to \eqref{eq_simple} with $u_0=0$ and $g=0$ such that $u\in \bL_{p,\theta}(\tau)$ and $x_d^{-1}D_d u \in \bL_{p,\theta}(\tau)$.
\end{lemma}

\begin{proof}
    Let us extend $f$ to $x_d<0$ by even extension, which yields that
    \begin{equation*}
        f\in L_p((0,\tau],\cP; L_p(\bR^d, |x_d|^{\theta-d} dx)).
    \end{equation*}
    Then it can be shown that
    \begin{equation*}
        u(t,x) = \int_0^t \int_{\bR^d} p(t-s,x-y) f(s,y) dyds = \int_0^t \int_{\bR^d_+} p_{\cN}(t-s,x,y) f(s,y) dyds
    \end{equation*}
    is a distributional solution to \eqref{eq_simple} on $\bR^d$, and consequently on $\bR^d_+$ as well.
    
    Since $|x_d|^{\theta-d}$ is in the $A_p(\bR)$ Muckenhoupt class of weights, by using \cite[Theorem 3.4]{HKP20}, one can easily deduce that for each $\omega\in\Omega$,
    \begin{align*}
        \int_0^\tau \int_{\bR^d} |u(\omega,t,x)|^p |x_d|^{\theta-d} dxdt \leq N \int_0^\tau \int_{\bR^d} |f(\omega,t,x)|^p |x_d|^{\theta-d} dxdt.
    \end{align*}
    Taking expectations on both sides, we get
    \begin{equation*}
        \|u\|_{\bL_{p,\theta}(\tau)} \leq N \|f\|_{\bL_{p,\theta}(\tau)}.
    \end{equation*}

    Now it remains to show $x_d^{-1}D_d u \in \bL_{p,\theta}(\tau)$. Since $f\in \bH_{0}^{\infty}(\tau)$,
    \begin{align*}
        D_du(t,x) &= \int_0^t \int_{\bR^d_+} D_{x_d} p_{\cN}(t-s,x,y) f(s,y) dyds 
        \\
        &= -\int_0^t \int_{\bR^d_+} D_{y_d} p_{\cD}(t-s,x,y) f(s,y) dyds
        \\
        &= \int_0^t \int_{\bR^d_+} p_{\cD}(t-s,x,y) D_{y_d}f(s,y) dyds,
    \end{align*}
    where $p_{\cD}(t,x,y) := p(t,x-y)-p(t,x-y^*)$. Thus, by \cite[Theorem 1]{KN14},
    \begin{equation*}
        \|x_d^{-1}D_du\|_{\bL_{p,\theta}(\tau)} \leq N \|f\|_{\bL_{p,\theta}(\tau)},
    \end{equation*}
    which obviously proves the desired claim.
    The lemma is proved.
\end{proof}

\begin{lemma} \label{lem4161852}
    Let $p\in[2,\infty)$, $\theta\in(d-1,d-1+p)$, $T\in (0,\infty)$, $\tau\leq T$ be a bounded stopping time, and $g\in \overline{\bH}_0^{\infty}(\tau,l_2)$. Then
    \begin{equation} \label{eq4161849}
        u(t,x) := \sum_{k=1}^\infty \int_0^t \int_{\bR^d_+} p_{\cN}(t-s,x,y) g^k(s,y) dy dw_s^k
    \end{equation}
    is in the space $\bL_{p,\theta}(\tau)$, and there is $N=N(d,p,\theta)$ so that
    \begin{equation} \label{eq4161514}
        \|u\|_{\bL_{p,\theta}(\tau)} \leq NT^{\frac{1}{2}} \|g\|_{\bL_{p,\theta}(\tau,l_2)}.
    \end{equation}
\end{lemma}

\begin{proof}
Let $\beta:=\frac{\theta-d}{p}$.
Since $p_{\cN}(t,x,y) \leq 2p(t,x-y)$, by the Burkholder-Davis-Gundy inequality and H\"older's inequality,
\begin{align*}
\bE|u(t,x)|^p &\leq N \bE \Bigg[\int_0^t \left(\int_{\bR^d_+} p(t-s,x-y)  R_{t-s,y}^{-\beta p'} dy\right)^{2/p'}
\\
&\qquad \quad \times \left(\int_{\bR^d_+} p(t-s,x-y) R_{t-s,y}^{\beta p} |g(s,y)|_{l_2}^p dy \right)^{2/p} ds\Bigg]^{p/2}
\\
&=: N \bE \left[ \int_0^t I^{2/p'}(t,s,x) \times II^{2/p}(t,s,x) ds \right]^{p/2},
\end{align*}
where $1/p+1/p'=1$ and $R_{t,x}:=\frac{x_d}{x_d+\sqrt{t}}$.
Again by H\"older's inequality,
\begin{align*}
\bE|u(t,x)|^p &\leq N \|I(t,\cdot,x)\|_{L_{2(p-1)/(p-2)}((0,t))}^{p-1} \times \bE\left[ \int_0^t II(t,s,x) ds\right].
\end{align*}
Here, we use the common convention $1/0:=\infty$ for $p=2$. 

Let us estimate $\|I(t,\cdot,x)\|_{L_{2(p-1)/(p-2)}((0,t))}^{p-1}$. 
By \eqref{eq3102020} with $\gamma_0=\gamma_1=-\beta p' >-1$, for each $0<s<t<T$ and $x\in \bR^d_+$, we see that
\begin{equation*}
    I(t,s,x)\leq N,
\end{equation*}
which yields
\begin{align} \label{eq4161511}
\|I(t,\cdot,x)\|_{L_{2(p-1)/(p-2)}((0,t))}^{p-1} \leq N T^{\frac{p-2}{2}}.
\end{align}

By \eqref{eq4161511} and Fubini's theorem,
\begin{align} \label{eq4161513}
&\bE\int_0^{\tau} \int_{\bR^d_+} |u(t,x)|^p x_d^{\theta-d} dxdt \leq N T^{\frac{p-2}{2}} \bE \int_0^{\tau} \int_{\bR^d_+} \left[ \int_0^t II(t,s,x) ds\right] x_d^{\theta-d} dtdx \nonumber
\\
&\leq N T^{\frac{p-2}{2}}\bE\int_0^{\tau} \int_{\bR^d_+} |g(s,y)|_{l_2}^p \times\left(\int_s^{\tau} \int_{\bR^d_+} p(t-s,x-y) R_{t-s,y}^{\beta p} x_d^{\theta-d} dx dt\right) dy ds.
\end{align}
By \eqref{eq3102020} with $\gamma_0= \theta-d$ and $\gamma_1=0$, we get
\begin{align*}
&\int_s^{\tau} \int_{\bR^d_+} p(t-s,x-y) R_{t-s,y}^{\beta p} x_d^{\theta-d} dx dt 
\\
&\leq N y_d^{\beta p}\int_s^T (\sqrt{t-s}+y_d)^{\theta-d-\beta p} dt \leq N T y_d^{\theta-d}.
\end{align*}
This and \eqref{eq4161513} yield \eqref{eq4161514}. The lemma is proved.
\end{proof}

\begin{lemma} \label{lem4162006}
    Let $p\in[2,\infty)$, $\theta\in(d-1,d-1+2p)$, $T\in (0,\infty)$, $\tau\leq T$ be a bounded stopping time, and $g\in \overline{\bH}_0^{\infty}(\tau,l_2)$. Then
    \begin{equation} \label{eq5262329}
        u(t,x) := \sum_{k=1}^\infty \int_0^t \int_{\bR^d_+} p_{\cD}(t-s,x,y) g^k(s,y) dy dw_s^k
    \end{equation}
    is in the space $\bL_{p,\theta-p}(\tau)$, and there is $N=N(d,p,\theta)$ so that
    \begin{equation} \label{eq3121250}
        \|x_d^{-1}u\|_{\bL_{p,\theta}(\tau)} \leq N \|g\|_{\bL_{p,\theta}(\tau,l_2)}.
    \end{equation}
\end{lemma}

\begin{proof}
For $p=2$, we denote $\beta_1:=1$, and for $p>2$ we choose $\beta_1$ satisfying
\begin{align} \label{eq3121234}
-\frac{2}{p} < - \beta_1 < \frac{\theta-d-1}{p}.
\end{align}
Since $\theta-d>-1$, we can take such $\beta_1$ for $p>2$, and the second inequality holds if $p=2$.
Due to $\theta-d<2p-1$, we can also take $\beta_2$ satisfying
\begin{align*}
\frac{\theta-d}{p} < \beta_2 < 2-\frac{1}{p}.
\end{align*}
Let us notice that
    \begin{equation} \label{eq4161845}
        p_{\cD}(t,x,y) \leq N \frac{R_{t,x}R_{t,y}}{t^{d/2}}e^{-\frac{|x-y|^2}{4\sigma t}} = NR_{t,x}R_{t,y}p(\sigma t,x-y),
    \end{equation}
    for some $\sigma>0$, where
\begin{equation*}
    R_{t,x}:=\frac{x_d}{x_d+\sqrt{t}}, \quad R_{t,y}:=\frac{y_d}{y_d+\sqrt{t}}.
\end{equation*}
 By the Burkholder-Davis-Gundy inequality, \eqref{eq4161845}, and H\"older's inequality,
\begin{align*}
&\bE|u(t,x)|^p 
\\
&\leq N \bE \Bigg[\int_0^t \left(\int_{\bR^d_+} p(\sigma(t-s),x-y) R_{t-s,x}^{(1-\beta_1) p'} R_{t-s,y}^{(1-\beta_2)p'} dy\right)^{2/p'}
\\
&\qquad \quad \times \left(\int_{\bR^d_+} p(\sigma(t-s),x-y) R_{t-s,x}^{\beta_1 p} R_{t-s,y}^{\beta_2p} |g(s,y)|_{l_2}^p dy \right)^{2/p} ds\Bigg]^{p/2}
\\
&=: N \bE \left[ \int_0^t I^{2/p'}(t,s,x) \times II^{2/p}(t,s,x) ds \right]^{p/2},
\end{align*}
where $1/p+1/p'=1$.
Again by H\"older's inequality, we actually have
\begin{align} \label{2112271503}
\bE|u(t,x)|^p &\leq N \|I(t,\cdot,x)\|_{L_{2(p-1)/(p-2)}((0,t))}^{p-1} \times \bE \left[ \int_0^t II(t,s,x) ds\right].
\end{align}
Here, we use the common convention $1/0:=\infty$ for $p=2$. 

Let us estimate $\|I(t,\cdot,x)\|_{L_{2(p-1)/(p-2)}((0,t))}^{p-1}$. 
By \eqref{eq3102020} with $\gamma_0 = \gamma_1 = (1-\beta_2)p'$, for each $0<s<t<T$ and $x\in \bR^d_+$, we see that
\begin{equation*}
    I(t,s,x)\leq N R_{t-s,x}^{(1-\beta_1)p'}.
\end{equation*}
For $p=2$, using $\beta_1=1$, we obtain
\begin{align} \label{2112271453}
\|I(t,\cdot,x)\|_{L_{2(p-1)/(p-2)}((0,t))}^{p-1} \leq N.
\end{align}
For $p>2$, by the change of variables,
\begin{align*}
&\int_0^t I^{2(p-1)/(p-2)}(t,s,x)ds 
\\
&\leq N x_d^{2(1-\beta_1)p/(p-2)} \int_0^t (\sqrt{t-s}+x_d)^{2(\beta_1-1)p/(p-2)} ds
\\
&\leq N x_d^{2} \int_0^\infty (\sqrt{s}+1)^{2(\beta_1-1)p/(p-2)} ds \leq Nx_d^{2}.
\end{align*}
Here, the last integral is finite due to \eqref{eq3121234}.
This and \eqref{2112271453} yield that for $p\geq 2$,
\begin{align*}
\|I(t,\cdot,x)\|_{L_{2(p-1)/(p-2)}((0,t))}^{p-1} \leq N x_d^{p-2}.
\end{align*}

By \eqref{2112271503} and Fubini's theorem,
\begin{align} \label{2112271505}
&\bE\int_0^{\tau} \int_{\bR^d_+} |u(t,x)|^p x_d^{\theta-d-p} dxdt \nonumber
\\
&\leq N \bE \int_0^{\tau} \int_{\bR^d_+} \left[ \int_0^t II(t,s,x) ds\right] x_d^{\theta-d-2} dtdx \nonumber
\\
&\leq N \bE\int_0^{\tau} \int_{\bR^d_+} |g(s,y)|_{l_2}^p \nonumber
\\
&\qquad\qquad \times\left(\int_s^{\tau} \int_{\bR^d_+} p(\sigma(t-s),x-y) R_{t-s,x}^{\beta_1p} R_{t-s,y}^{\beta_2p} x_d^{\theta-d-2} dx dt\right) dy ds.
\end{align}
By \eqref{eq3102020} with $\gamma_0= \theta-d-2+\beta_1p$ and $\gamma_1=\beta_1p$, we get
\begin{align*}
&\int_s^{\tau} \int_{\bR^d_+} p(\sigma(t-s),x-y) R_{t-s,x}^{\beta_1p} R_{t-s,y}^{\beta_2p} x_d^{\theta-d-2} dx dt
\\
&\leq N \int_s^T (\sqrt{t-s}+y_d)^{\theta-d- 2 - \beta_2p} y_d^{\beta_2 p} dt
\\
&\leq N \int_0^\infty (\sqrt{t}+1)^{\theta-d- 2 - \beta_2p} y_d^{\theta-d} dt \leq N y_d^{\theta-d}.
\end{align*}
This and \eqref{2112271505} yield \eqref{eq3121250}. The lemma is proved.
\end{proof}

\begin{remark}
    In fact, \eqref{eq5262329} is the representation formula for solutions to \eqref{eq_simple} under the Dirichlet condition. When $\theta\in(d-1,d-1+p)$, \eqref{eq3121250} is obtained in \cite{Kry_halfspace}. The difference in the admissible range of $\theta$ arises from the fact that \cite{Kry_halfspace} employed an integration by parts argument applicable even in the presence of an inhomogeneous term $f$. Our result implies that one may treat solutions with irregular boundary behavior when $f=0$.
\end{remark}

\begin{lemma} \label{lem4162024}
    Let $p\in[2,\infty)$, $\theta\in(d-1,d-1+p)$, $T\in(0,\infty)$, $\tau\leq T$ be a bounded stopping time, and $g\in \overline{\bH}_0^{\infty}(\tau,l_2)$. Then $u$ defined by \eqref{eq4161849} is a solution to \eqref{eq_simple} with $u_0=0$ and $f=0$ such that $u\in \bL_{p,\theta}(\tau)$ and $x_d^{-1}D_du\in \bL_{p,\theta}(\tau)$.
\end{lemma}

\begin{proof}
    By Lemma \ref{lem4161852}, one can easily find that $u \in \bL_{p,\theta}(\tau)$.
    For $D_du$, since $g\in\overline{\bH}_0^{\infty}(\tau,l_2)$ and $D_{x_d}p_{\cN}(t,x,y) = -D_{y_d} p_{\cD}(t,x,y)$,
    \begin{align*}
        D_du(t,x) &= - \sum_{k=1}^\infty \int_0^t \int_{\bR^d_+} D_{y_d}p_{\cD}(t-s,x,y) g^k(s,y) dy dw_s^k 
        \\
        &=  \sum_{k=1}^\infty \int_0^t \int_{\bR^d_+} p_{\cD}(t-s,x,y) D_{y_d}g^k(s,y) dy dw_s^k.
    \end{align*}
    Thus, $x_d^{-1}D_du\in \bL_{p,\theta}(\tau)$ follows from Lemma \ref{lem4162006}.
    By extending $g$ to $x_d<0$ by even extension, we get
    \begin{equation*}
        u(t,x) = \sum_{k=1}^\infty \int_0^t \int_{\bR^d} p(t-s,x-y) g^k(s,y) dy dw_s^k,
    \end{equation*}
    which is a solution to $du=\Delta udt + g^kdw_t^k$ in the whole space $\bR^d$ (see \cite[Lemma 4.2]{Kry99}). This obviously implies that $u$ satisfies \eqref{eq_simple} in the sense of distributions.
The lemma is proved.
\end{proof}

\section{A priori estimates} \label{sec_apriori}

In this section, our goal is to obtain the a priori estimate \eqref{main_est} for general coefficients. Throughout this section, $f$ and $g$ are independent of $u$.

We first prove an auxiliary result.

\begin{lemma}
    Let $p\in (1,\infty)$, $\theta\in(d-1,d-1+p)$, and $\tau$ be a bounded stopping time.
    Let $\varphi\in C_c^\infty(\bR)$ such that $\varphi\geq0$, $\int_{\bR}\varphi(x_d) dx_d=1$, and $\text{supp}(\varphi)\subset \{x_d<0\}$.
    For any $h\in \bL_{p,\theta}(\tau)$,
        \begin{equation} \label{eq9051117}
        \|h^{(\varepsilon)}\|_{\bL_{p,\theta}(\tau)} \leq N(d,p,\theta,\varphi) \|h\|_{\bL_{p,\theta}(\tau)},
    \end{equation}
where
\begin{equation*}
    h^{(\varepsilon)}(x) := \int_{0}^{\infty} \frac{1}{\varepsilon} \varphi\left(\frac{x_d-y_d}{\varepsilon}\right) h(x',y_d) dy_d, \quad \varepsilon>0.
\end{equation*}
\end{lemma}

\begin{proof}
    Suppose that $\text{supp}(\varphi) \subset (-c,0)$ for some $c>0$.
    \begin{equation*}
        |h^{(\varepsilon)}(t,x)| \leq \frac{N}{\varepsilon} \int_{x_d}^{x_d+c\varepsilon} |\tilde{h}(t,x',y_d)| dy_d \leq N \cM \tilde{h}(x),
    \end{equation*}
    where $\tilde{h}$ is the zero extension of $h$ to $x_d<0$, and $\cM \tilde{h}$ is the (one-sided) maximal function of $\tilde{h}$ defined as
    \begin{equation*}
        \cM \tilde{h}(t,x) := \sup_{r>0} \frac{1}{r} \int_{x_d}^{x_d+r} |\tilde{h}(t,x',y_d)| dy_d.
    \end{equation*}
    Since $|x_d|^{\theta-d}$ is in the $A_p(\bR)$ Muckenhoupt class of weights, by Hardy–Littlewood theorem for $A_p(\bR)$-weight (see \cite[Example 7.1.7 and Theorem 7.1.9]{G14}), one can see that 
    \begin{equation*}
        \int_{0}^{\infty} |\cM \tilde{h}(t,x',x_d)|^p x_d^{\theta-d} dx_d \leq N \int_{0}^{\infty} |\tilde{h}(t,x',x_d)|^p x_d^{\theta-d} dx_d.
    \end{equation*}
By integrating this inequality with respect to $(t,x')$ and taking expectations, we obtain the desired result.
The lemma is proved.
\end{proof}

\begin{lemma}[Estimation of $D_du$ and $D_xD_du$]
    Let $\gamma\geq\mu$, $p\in[2,\infty)$, $T\in(0,\infty)$, and $\tau\leq T$ be a bounded stopping time.
    Suppose that Assumption \ref{as_coeff} holds, and we assume that
    \begin{equation*}
        d-1+p\left( 1- \frac{1}{p(1-\delta_1)+\delta_1} \right)<\theta < d-1+p.
    \end{equation*}
    Assume that $f\in \bH_{p,\theta}^{\gamma}(\tau)$, $g\in {\bH}_{p,\theta}^{\gamma+1,(1)}(\tau,l_2)$, and $u$ is a $\cD'(\bR^d_+)$-valued predictable function defined on $(0,\tau]$ such that $x_d^{-1} D_du \in \bH_{p,\theta}^{\mu+1}(\tau)$. If $u$ is a solution to \eqref{maineq} with $u_0=0$, then we actually have $x_d^{-1} D_du \in \bH_{p,\theta}^{\gamma+1}(\tau)$ and $D_du \in L_p(C([0,\tau];H_{p,\theta}^{\gamma}))$. Moreover,
    \begin{align} \label{2511132025}
        &\left(\bE \sup_{s\leq \tau}\|D_du(s,\cdot)\|_{H_{p,\theta}^{\gamma}}^p\right)^{1/p} + \|x_d^{-1}D_du\|_{\bH_{p,\theta}^{\gamma+1}(\tau)} + \|D_xD_du\|_{\bH_{p,\theta}^{\gamma}(\tau)} \nonumber
        \\
        &\leq N \left( \|f\|_{\bH_{p,\theta}^{\gamma}(\tau)} + \|g\|_{{\bH}_{p,\theta}^{\gamma+1,(1)}(\tau,l_2)} \right),
    \end{align}
    where $N=N(d,\gamma,p,\theta,\delta_0,\delta_1,T)$.
\end{lemma}

\begin{proof}
Since by \eqref{equiv_der},
\begin{equation*}
    \|D_xD_du\|_{\bH_{p,\theta}^\gamma(\tau)} \leq N \|x_d^{-1}D_du\|_{\bH_{p,\theta}^{\gamma+1}(\tau)}
\end{equation*}
provided that $x_d^{-1}D_du \in \bH_{p,\theta}^{\gamma+1}(\tau)$, we only need to estimate
\begin{equation*}
    \left(\bE \sup_{s\leq \tau}\|D_du(s,\cdot)\|_{H_{p,\theta}^{\gamma}}^p\right)^{1/p} + \|x_d^{-1}D_du\|_{\bH_{p,\theta}^{\gamma+1}(\tau)}.
\end{equation*}
    Note that $v:=D_du \in \bH_{p,\theta-p}^{\mu+1}(\tau)$ satisfies
    \begin{equation} \label{eq4141844}
        dv = \left( a^{ij}(t)D_{ij} v + D_df \right) dt + \left( \sigma^{ik}(t) D_i v +  D_d g^k \right) dw_t^k, \quad t>0, x\in \bR^d_+
    \end{equation}
    with the initial data $v(0,x)=0$ in the sense of distributions. Roughly, one can regard that $v$ satisfies the equation with the Dirichlet boundary condition $v(t,x',0)=0$.
    We note that $v$, $D_df$, and $D_dg$ satisfy the conditions of \cite[Theorem 3.3]{Kry_halfspace}. Thus, using this theorem, there is a solution $w\in \bH_{p,\theta-p}^{\gamma+1}(\tau)\subset \bH_{p,\theta-p}^{\mu+1}(\tau)$ to \eqref{eq4141844}. Moreover, the uniqueness result in the space $\bH_{p,\theta-p}^{\mu+1}(\tau)$ implies that $v=w\in \bH_{p,\theta-p}^{\gamma+1}(\tau)$. By \cite[Theorems 2.11 and 3.3]{Kry_halfspace},
    \begin{align*}
        &\left(\bE \sup_{s\leq \tau}\|D_du(s,\cdot)\|_{H_{p,\theta}^{\gamma}}^p\right)^{1/p} + \|x_d^{-1}D_du\|_{\bH_{p,\theta}^{\gamma+1}(\tau)}
        \\
        &\leq N \left( \|x_dD_df\|_{\bH_{p,\theta}^{\gamma-1}(\tau)} + \|D_dg\|_{\bH_{p,\theta}^{\gamma}(\tau,l_2)} \right)
        \\
        &\leq N \left( \|f\|_{\bH_{p,\theta}^{\gamma}(\tau)} + \|g\|_{{\bH}_{p,\theta}^{\gamma+1,(1)}(\tau,l_2)} \right).
    \end{align*}
    For the last inequality, we used \eqref{equiv_der}. The proof is completed.
\end{proof}

\begin{lemma}[Estimation of $D_{x'}u$ when $\gamma\geq0$] \label{lem4071423}
Let $\gamma\geq0$, $p\in[2,\infty)$, $T\in(0,\infty)$, and $\tau\leq T$ be a bounded stopping time.
    Suppose that Assumption \ref{as_coeff} holds, and we assume that
    \begin{equation*}
        d-1+p\left( 1- \frac{1}{p(1-\delta_1)+\delta_1} \right)<\theta < d-1+p.
    \end{equation*}
    Assume that $f\in \bH_{p,\theta}^{\gamma}(\tau)$, $g\in {\bH}_{p,\theta}^{\gamma+1,(1)}(\tau,l_2)$, $u\in \bL_{p,\theta}(\tau)$, and $x_d^{-1}D_du \in \bL_{p,\theta}(\tau)$. If $u$ is a solution to \eqref{maineq} with $u_0=0$, then for $m\in\bN_0$ such that $m\leq \gamma$,
\begin{align} \label{eq2511141208}
    &\|D_{x'}^mu\|_{\bL_{p,\theta+mp}(\tau)} + \|D_{x'}^{m+1} u\|_{\bL_{p,\theta+mp}(\tau)} + \|D_{x'}^{m+2} u\|_{\bL_{p,\theta+mp}(\tau)} \nonumber
    \\
    &\leq N\left( \|f\|_{\bH_{p,\theta}^{\gamma}(\tau)} + \|g\|_{{\bH}_{p,\theta}^{\gamma+1,(1)}(\tau,l_2)} \right),
\end{align}
where $N=N(d,\gamma,p,\theta,\delta_0,\delta_1,T)$.
\end{lemma}

\begin{proof}
Let $\varphi\in C_c^\infty(\bR)$ such that $\varphi\geq0$, $\int_{\bR} \varphi(x_d) dx_d=1$, and $\text{supp}(\varphi) \subset \{x_d<0\}$.
For $x_d>0$, $\varepsilon>0$, and any measurable function $h$ defined on $\bR^d_+$, we denote
\begin{equation*}
    h^{(\varepsilon)}(x) = \int_{0}^{\infty} \frac{1}{\varepsilon} \varphi\left(\frac{x_d-y_d}{\varepsilon}\right)h(x',y_d) dy_d.
\end{equation*}
Moreover, one can see that for any $x_d>0$ and $\varphi^\varepsilon(x_d):=\varepsilon^{-1}\varphi(x_d/\varepsilon)$ with $\phi\in C_c^\infty(\bR^{d-1})$, $y\to\varphi^{\varepsilon}(x_d-y_d)\phi(y') \in C_c^\infty(\bR^d_+)$, and thus
\begin{equation*}
    (h,\varphi^{\varepsilon}(x_d-\cdot)\phi(\cdot))_{\bR^d_+} = (h^{(\varepsilon)}(\cdot,x_d),\phi)_{\bR^{d-1}}.
\end{equation*}
In fact, this mollification can be extended to $\cD'(\bR^d_+)$. Moreover, if $h\in L_{p,\theta}$, then for each $x_d>0$ and $\varepsilon>0$,
\begin{align*}
    |(x_d^m D_{x'}^mh,\varphi^{\varepsilon}(x_d-\cdot)\phi(\cdot))_{\bR^d_+}| &\leq N(x_d,\varepsilon) \|h\|_{L_{p,\theta}} \|D_{x'}^m\phi\|_{L_{p'}(\bR^{d-1})}
    \\
    &\leq N(x_d,\varepsilon) \|h\|_{L_{p,\theta}} \|\phi\|_{H_{p'}^{m}(\bR^{d-1})},
\end{align*}
where $1/p+1/p'=1$. This yields $(x_d^mD_{x'}^mh)^{(\varepsilon)}(\cdot,x_d) \in H_p^{-m}(\bR^{d-1})$.
Hence, letting $v:=x_d^mD_{x'}^mu$, for a fixed $x_d>0$,
\begin{align} \label{eq4152213}
    dv^{(\varepsilon)} &= \left( \sum_{i,j<d} a^{ij} D_{ij} v^{(\varepsilon)} + (x_d^mD_{x'}^m f)^{(\varepsilon)} + \sum_{i=1}^d (a^{id}+a^{di}1_{i<d})(x_d^mD_{x'}^mD_{di}u)^{(\varepsilon)} \right) dt \nonumber
    \\
    &\, + \left( \sum_{i=1}^{d-1}\sigma^{ik} D_i v^{(\varepsilon)} + (x_d^mD_{x'}^mg^k)^{(\varepsilon)} + \sigma^{dk}(x_d^mD_{x'}^mD_du)^{(\varepsilon)} \right) dw_t^k,
\end{align}
with $v^{(\varepsilon)}(0,\cdot,x_d) = 0$ in the sense of distributions in $\bR^{d-1}$. Here, all functions in \eqref{eq4152213} are regarded as functions in the variables $(t,x')$.
By applying \cite[Theorems 3.7 and 4.10]{Kry99} with respect to $(t,x')$,
\begin{align} \label{eq4062152}
    &\|v^{(\varepsilon)}\|_{\bL_p(\bR^{d-1},\tau)}^p + \|D_{x'} v^{(\varepsilon)}\|_{\bL_p(\bR^{d-1},\tau)}^p + \|D_{x'}^2 v^{(\varepsilon)}\|_{\bL_p(\bR^{d-1},\tau)}^p \nonumber
    \\
    &\leq N \sum_{i=1}^{d} \|(x_d^mD_{x'}^mD_{di} u)^{(\varepsilon)}\|_{\bL_p(\bR^{d-1},\tau)}^p + N \|(x_d^mD_{x'}^mD_du)^{(\varepsilon)}\|_{\bH_p^{1}(\bR^{d-1},\tau)}^p \nonumber
    \\
    &\quad + N\|(x_d^mD_{x'}^mf)^{(\varepsilon)}\|_{\bL_p(\bR^{d-1},\tau)}^p + N\|(x_d^mD_{x'}^mg)^{(\varepsilon)}\|_{\bL_p(\bR^{d-1},\tau,l_2)}^p \nonumber
    \\
    &\quad+ N\|(x_d^mD_{x'}^{m+1}g)^{(\varepsilon)}\|_{\bL_p(\bR^{d-1},\tau,l_2)}^p.
\end{align}
Here, the spaces are defined by
\begin{equation*}
    \bL_{p}(\bR^{d-1},\tau) := L_p((0,\tau],\cP;L_p(\bR^{d-1})), \quad \bH_{p}^1(\bR^{d-1},\tau) := L_p((0,\tau],\cP;H_p^1(\bR^{d-1})),
\end{equation*}
and
\begin{equation*}
    \bL_p(\bR^{d-1},\tau,l_2) := L_p((0,\tau],\cP;L_p(\bR^{d-1},l_2)).
\end{equation*}
Due to \eqref{2511132025}, the original functions on the right-hand side of \eqref{eq4062152} before mollification belong to $\bL_{p,\theta}(\tau)$, for example, $x_d^mD_{x'}^mD_{di}u \in \bL_{p,\theta}(\tau)$. For $h\in \bL_{p,\theta}(\tau)$ and $x_d,\bar{x}_d\in[a,b]$ such that $0<a<b<\infty$, one can show that
\begin{equation*}
    \left\|h^{(\varepsilon)}(x_d) - h^{(\varepsilon)}(\bar{x}_d)\right\|_{\bL_{p}(\bR^{d-1},\tau)} \leq N(\varepsilon,a,b,\theta)|x_d-\bar{x}_d| \|h\|_{\bL_{p,\theta}(\tau)}.
\end{equation*}
Hence, by replacing $v^{(\varepsilon)}(x_d)$ with $v^{(\varepsilon)}(x_d)-v^{(\varepsilon)}(\bar{x}_d)$ in \eqref{eq4062152}, $x_d\to v^{(\varepsilon)}(x_d)$ is continuous with values in $\bL_{p}(\bR^{d-1},\tau)$. Similarly, $D_{x'}v^{(\varepsilon)}(x_d)$ and $D_{x'}^2v^{(\varepsilon)}(x_d)$ are also continuous. This allows us to multiply both sides of \eqref{eq4062152} by $x_d^{\theta-d}$ and integrate with respect to $x_d$.
Thus, by Fubini's theorem, \eqref{eq9051117}, and \eqref{2511132025},
\begin{align*}
    &\|v^{(\varepsilon)}\|_{\bL_{p,\theta}(\tau)} + \|D_{x'} v^{(\varepsilon)}\|_{\bL_{p,\theta}(\tau)} + \|D_{x'}^2 v^{(\varepsilon)}\|_{\bL_{p,\theta}(\tau)}
    \\
    &\leq N \|x_d^mD_{x'}^{m+1}D_d u\|_{\bL_{p,\theta}(\tau)} + N \|x_d^mD_{x'}^m D_{d} u\|_{\bL_{p,\theta}(\tau)} +  N \|x_d^mD_{x'}^m D_{d}^2 u\|_{\bL_{p,\theta}(\tau)}
     \\
     &\quad+ N\|x_d^mD_{x'}^mf\|_{\bL_{p,\theta}(\tau)} + N \|x_d^mD_{x'}^{m} g\|_{\bL_{p,\theta}(\tau,l_2)} + N \|x_d^mD_{x'}^{m+1} g\|_{\bL_{p,\theta}(\tau,l_2)}
     \\
     &\leq N\|x_d^{-1}D_du\|_{\bH_{p,\theta}^{m+1}(\tau)} + N\|D_du\|_{\bH_{p,\theta}^m(\tau)} + N\|D_d^2u\|_{\bH_{p,\theta}^m(\tau)}
     \\
     &\quad+ N \|f\|_{\bH_{p,\theta}^m(\tau)} + N \|g\|_{{\bH}_{p,\theta}^{m+1,(1)}(\tau,l_2)}
     \\
     &\leq N\left( \|f\|_{\bH_{p,\theta}^{\gamma}(\tau)} + \|g\|_{{\bH}_{p,\theta}^{\gamma+1,(1)}(\tau,l_2)} \right).
\end{align*}
Here, we emphasize that the right-hand side is independent of $\varepsilon$. By Lemma \ref{lem4161836}~$(vi)$, $\bL_{p,\theta}(\tau)$ is reflexive.  This implies that there is a sequence $\varepsilon_n\to0$ such that for each $k=0,1,2$, $D_{x'}^k v^{(\varepsilon_n)}$ converges weakly in $\bL_{p,\theta}(\tau)$ as $\varepsilon_n\to0$. Since for each $(t,\omega)$, $v^{(\varepsilon_n)}(t,\cdot)\to x_d^mD_{x'}^m u$ in $\cD'(\bR^d_+)$ as $\varepsilon_n\to0$, we have
\begin{align*}
    &\|x_d^mD_{x'}^mu\|_{\bL_{p,\theta}(\tau)} + \|x_d^mD_{x'}^{m+1} u\|_{\bL_{p,\theta}(\tau)} + \|x_d^mD_{x'}^{m+2} u\|_{\bL_{p,\theta}(\tau)} 
     \\
     &\leq N\left( \|f\|_{\bH_{p,\theta}^{\gamma}(\tau)} + \|g\|_{{\bH}_{p,\theta}^{\gamma+1,(1)}(\tau,l_2)} \right),
\end{align*}
which yields \eqref{eq2511141208}.
The lemma is proved.
\end{proof}

\begin{lemma}[A priori estimate when $\gamma\in\bN_0$] \label{lem4162028}
    Let the assumptions of Lemma \ref{lem4071423} be satisfied for $\gamma\in\bN_0$. Then we have
        \begin{align} \label{eq2511151744}
    \|u\|_{\frH_{neu}^{\gamma+2}(p,\theta,\tau)} &\leq N\left( \|f\|_{\bH_{p,\theta}^{\gamma}(\tau)} + \|g\|_{{\bH}_{p,\theta}^{\gamma+1,(1)}(\tau,l_2)} \right),
\end{align}
where $N=N(d,\gamma,p,\theta,\delta_0,\delta_1,T)$.
\end{lemma}

\begin{proof}
By \eqref{def_norm} and the assumptions $f\in \bH_{p,\theta}^{\gamma}(\tau)$ and $g\in {\bH}_{p,\theta}^{\gamma+1,(1)}(\tau,l_2)$, it suffices to estimate
\begin{equation*}
    \|u\|_{\bH^{\gamma+2,(2)}_{p,\theta}(\tau)} + \|x_d^{-1}D_du\|_{\bH_{p,\theta}^{\gamma+1}(\tau)}.
\end{equation*}
Since the estimation of $\|x_d^{-1}D_du\|_{\bH_{p,\theta}^{\gamma+1}(\tau)}$ is already obtained in \eqref{2511132025}, we only need to estimate $\|u\|_{\bH^{\gamma+2,(2)}_{p,\theta}(\tau)}$.

First, we consider $D_du$ and $D_xD_du$.
    Note that
        \begin{equation*}
        \|D_d u\|_{\bH_{p,\theta}^{\gamma}(\tau)} \leq N(T) \left(\bE \sup_{s\leq \tau}\|D_du(s,\cdot)\|_{H_{p,\theta}^{\gamma}}^p\right)^{1/p},
    \end{equation*}
    and by \eqref{equiv_der},
    \begin{equation*}
       \|D_xD_du\|_{\bH_{p,\theta}^{\gamma}(\tau)} \leq N \|x_d^{-1}D_du\|_{\bH_{p,\theta}^{\gamma+1}(\tau)}.
    \end{equation*}
    Thus, $D_du$ and $D_xD_du$ can be estimated by \eqref{2511132025}. For $D_{x'}u$, and $D_{x'}^2u$, by \eqref{equiv_der} and \eqref{eq2511141208},
    \begin{align*}
        &\|D_{x'}u\|_{\bH_{p,\theta}^{\gamma}(\tau)} + \|D_{x'}^2u\|_{\bH_{p,\theta}^{\gamma}(\tau)}
        \\
        &\leq N \|D_du\|_{\bH_{p,\theta}^{\gamma}(\tau)} + N \|D_xD_du\|_{\bH_{p,\theta}^{\gamma}(\tau)}
        \\
        &\quad+ N \sum_{m=0}^{\gamma} \left(   \|D_{x'}^mu\|_{\bL_{p,\theta+mp}(\tau)} + \|D_{x'}^{m+1} u\|_{\bL_{p,\theta+mp}(\tau)} + \|D_{x'}^{m+2} u\|_{\bL_{p,\theta+mp}(\tau)}\right)
    \\
    &\leq N\left( \|f\|_{\bH_{p,\theta}^{\gamma}(\tau)} + \|g\|_{\bH_{p,\theta}^{\gamma+1,(1)}(\tau,l_2)} \right).
    \end{align*}
For $u$, by \eqref{eq2152119} and
     \begin{equation*}
        \|u\|_{\bH_{p,\theta}^{\gamma}(\tau)}^p \leq N(T) \bE \sup_{s\leq \tau}\|u(s,\cdot)\|_{H_{p,\theta}^{\gamma}}^p,
    \end{equation*}
    we have
    \begin{align*}
        \|u\|_{\bH_{p,\theta}^{\gamma}(\tau)} &\leq N \left(\|D_xu\|_{\bH_{p,\theta}^{\gamma}(\tau)} + \|D_x^2u\|_{\bH_{p,\theta}^{\gamma}(\tau)}  + \|f\|_{\bH_{p,\theta}^{\gamma}(\tau)} + \|g\|_{\bH_{p,\theta}^{\gamma}(\tau,l_2)}\right).
    \end{align*}
The proof is completed.
\end{proof}

\section{Proof of the main theorems} \label{sec_proof}

In this section, we provide the proofs of the main theorems in Section \ref{sec_res}.

\subsection{Proof of Theorem \ref{thm_main}}

\begin{lemma} \label{lem4162045}
Suppose that $u_0=0$, and $f$ and $g$ are independent of $u$. Then Theorem \ref{thm_main} holds when $\gamma \in \bN_0$.
\end{lemma}

\begin{proof}
    Due to the method of continuity and \eqref{main_est}, which follows from \eqref{eq2511151744}, we focus on the existence result when $a^{ij}(t)D_{ij} = \Delta$ and $\sigma^{ik}(t)\equiv0$. Assume that $f \in \bH_0^\infty(\tau)$ and $g\in \overline{\bH}_0^\infty(\tau,l_2)$. Then combining Lemmas \ref{lem4162023} and \ref{lem4162024}, there is a solution $u$ to \eqref{maineq} such that $u, x_d^{-1}D_du\in \bL_{p,\theta}(\tau)$. By Lemma \ref{lem4162028}, we actually have $u\in \frH_{neu}^{\gamma+2}(p,\theta,\tau)$. For general $f$ and $g$, one can use the approximation argument. The proof is completed.
\end{proof}

\begin{lemma} \label{lem5141719}
Suppose that $u_0=0$, and $f$ and $g$ are independent of $u$. Then Theorem \ref{thm_main} holds for general $\gamma \geq0$.
\end{lemma}

\begin{proof}
Since $\frH_{neu}^{\gamma+2}(p,\theta,\tau) \subset \frH_{neu}^{2}(p,\theta,\tau)$, the uniqueness result follows from Lemma \ref{lem4162045}. Thus, we only prove the existence of solution and \eqref{main_est}.

When $\gamma\in\bN_0$, given $(f,g) \in \bH_{p,\theta}^\gamma(\tau)\times {\bH}_{p,\theta}^{\gamma+1,(1)}(\tau,l_2)$, we define $\cT$ to be the mapping $(f,g)\to (u,x_d^{-1}D_du)$, where $u\in \frH_{neu}^{\gamma+2}(p,\theta,\tau)$ is a solution to \eqref{maineq}. Then by Lemma \ref{lem4162045}, 
\begin{align} \label{eq5271335}
    \cT: \bH_{p,\theta}^\gamma(\tau)\times {\bH}_{p,\theta}^{\gamma+1,(1)}(\tau,l_2) \to \bH_{p,\theta}^{\gamma+2,(2)}(\tau) \times \bH_{p,\theta}^{\gamma+1}(\tau)
\end{align}
is a well defined linear operator. Here, we remark that the operator can be taken to be independent of $\gamma$ due to the uniqueness result. 

Let $\gamma \geq 0$ be non-integer. For $n := \lfloor \gamma \rfloor$, take $\kappa\in (0,1)$ so that $\gamma = \kappa n + (1-\kappa)(n+1)$. By \eqref{eq5071804}, \eqref{eq5112341}, and \cite[Theorem 2.2.6]{HNVW16},
\begin{eqnarray*}
&[\bH^{n}_{p,\theta}(\tau), \bH^{n+1}_{p,\theta}(\tau)]_{\kappa} = \bH^{\gamma}_{p,\theta}(\tau),&
\\
&[\bH^{n+1,(1)}_{p,\theta}(\tau,l_2), \bH^{n+2,(1)}_{p,\theta}(\tau,l_2)]_{\kappa} = \bH^{\gamma+1,(1)}_{p,\theta}(\tau,l_2)&
\\
&[\bH^{n+2,(2)}_{p,\theta}(\tau), \bH^{n+3,(2)}_{p,\theta}(\tau)]_{\kappa} = \bH^{\gamma+2,(2)}_{p,\theta}(\tau), \quad [\bH^{n+1}_{p,\theta}(\tau), \bH^{n+2}_{p,\theta}(\tau)]_{\kappa} = \bH^{\gamma+1}_{p,\theta}(\tau).&
\end{eqnarray*}
Now we use \eqref{eq5112341} and the complex interpolation of operators to conclude that $\cT$ extends to a linear operator for general $\gamma\geq0$.
Thus, to conclude \eqref{main_est} as well as the solvability result, one just needs to show that $u$ is a solution to \eqref{maineq} in the sense of distributions. Indeed, since
\begin{equation*}
    \bH_{p,\theta}^\gamma(\tau)\times {\bH}_{p,\theta}^{\gamma+1,(1)}(\tau,l_2) \subset \bH_{p,\theta}^n(\tau)\times {\bH}_{p,\theta}^{n+1,(1)}(\tau,l_2)
\end{equation*}
for $n\in\bN_0$ so that $n\leq \gamma$, the definition of $\cT$ in \eqref{eq5271335} with the case $\gamma=n$ yields that the first component of $\cT (f,g)$, say $u \in \bH_{p,\theta}^{n+2,(2)}(\tau)$, is a solution to \eqref{maineq} in the sense of distributions.
The lemma is proved.
\end{proof}

\begin{lemma} \label{lem4201603}
Suppose that $u_0=0$, and $f$ and $g$ are independent of $u$. Then Theorem \ref{thm_main} holds for $\gamma <0$.
\end{lemma}

\begin{proof}
We first prove the existence result and the estimate \eqref{main_est}.
We proceed by induction on $n\in\bN_0$ such that $\gamma\in[-n,-n+1)$. When $n=0$, the case $\gamma\in[0,1)$ is already proved due to Lemma \ref{lem5141719}. Suppose that the lemma holds true for $n=0,1,\dots, m-1$ with $m\in\bN$. Then our goal is to prove the lemma when $n=m$.

By Lemmas \ref{lem4282020} and \ref{lem4161837} $(ii)$, there is $c=c(d,p,\theta,\gamma)>0$ such that
    \begin{equation*}
        \cL u:= x_d^2 \Delta u -cu
    \end{equation*}
    is a bounded isomorphism from $\bH_{p,\theta_0}^{\gamma_0+2,(n_0)}(\tau)$ onto $\bH_{p,\theta_0}^{\gamma_0,(n_0)}(\tau)$, for any choice of $\gamma_0 \in \{\gamma,\gamma+1,\gamma+2\}$, $\theta_0\in\{\theta,\theta-p\}$, and $n_0\in \{0,1,2\}$.
 Then we get
    \begin{equation*}
        \cL^{-1}f \in \bH_{p,\theta}^{\gamma+2}(\tau), \quad \cL^{-1}g \in \bH_{p,\theta}^{\gamma+3,(1)}(\tau,l_2).
    \end{equation*}
By the induction hypothesis, there is $v\in \frH_{neu}^{\gamma+4}(p,\theta,\tau)$, which is a solution to
\begin{equation*}
    dv= \left( a^{ij}D_{ij} v + \cL^{-1}f \right) dt + \left( \sigma^{ik} D_i v +  \cL^{-1}g^k \right) dw_t^k
\end{equation*}
with the zero initial data.
Then one can check that $\tilde{v}:=\cL v\in \frH_{neu}^{\gamma+2}(p,\theta,\tau)$ satisfies
\begin{align*}
    d\tilde{v} &= \left( a^{ij}D_{ij} \tilde{v} - 2x_d(a^{id}+a^{di})D_i\Delta v - 2a^{dd}\Delta v + f \right) dt 
    \\
    &\quad + \left( \sigma^{ik} D_i \tilde{v} -2x_d\sigma^{dk}\Delta v + g^k \right) dw_t^k
\end{align*}
with the zero initial condition, and
\begin{equation} \label{eq4170033}
    \|v\|_{\frH_{neu}^{\gamma+4}(p,\theta,\tau)} \approx \|\tilde{v}\|_{\frH_{neu}^{\gamma+2}(p,\theta,\tau)}.
\end{equation}
Since 
\begin{equation*}
    - 2x_d(a^{id}+a^{di})D_i\Delta v - 2a^{dd}\Delta v \in \bH_{p,\theta}^{\gamma+1}(\tau)
\end{equation*}
and
\begin{equation*}
    -2x_d\sigma^{dk}\Delta v \in{\bH}_{p,\theta}^{\gamma+2,(1)}(\tau,l_2),
\end{equation*}
there is $w\in \frH_{neu}^{\gamma+3}(p,\theta,\tau)$ such that
\begin{align*}
    dw &= \left( a^{ij}D_{ij} w - 2x_d(a^{id}+a^{di})D_i\Delta v - 2a^{dd}\Delta v \right) dt 
    \\
    &\quad + \left( \sigma^{ik} D_i w -2x_d\sigma^{dk}\Delta v \right) dw_t^k
\end{align*}
with the zero initial data $w(0,\cdot)=0$.
Now $u := \tilde{v}-w$ is the desired solution to \eqref{maineq} with $u_0=0$. Moreover, by \eqref{eq4170033},
\begin{align*}
    \|u\|_{\frH_{neu}^{\gamma+2}(p,\theta,\tau)} &\leq N \left(\|\tilde{v}\|_{\frH_{neu}^{\gamma+2}(p,\theta,\tau)} + \|w\|_{\frH_{neu}^{\gamma+3}(p,\theta,\tau)}\right)
    \\
    &\leq N\left( \| \cL^{-1}f\|_{\bH_{p,\theta}^{\gamma+2}(\tau)} + \|\cL^{-1}g\|_{{\bH}_{p,\theta}^{\gamma+3,(1)}(\tau,l_2)} \right)
    \\
    &\quad+ N \left(\| x_d(a^{id}+a^{di})D_i\Delta v + a^{dd}\Delta v\|_{\bH_{p,\theta}^{\gamma+1}(\tau)} + \|x_d\sigma^{dk}\Delta v\|_{{\bH}_{p,\theta}^{\gamma+2,(1)}(\tau,l_2)} \right)
    \\
    &\leq N\left( \|\cL^{-1}f\|_{\bH_{p,\theta}^{\gamma+2}(\tau)} + \|\cL^{-1}g\|_{{\bH}_{p,\theta}^{\gamma+3,(1)}(\tau,l_2)} \right)
    \\
    &\leq N\left(\|f\|_{\bH_{p,\theta}^{\gamma}(\tau)} + \|g\|_{{\bH}_{p,\theta}^{\gamma+1,(1)}(\tau,l_2)} \right).
\end{align*}
Hence, we obtain \eqref{main_est}.

Now we consider the uniqueness of solution. Let $u\in \frH_{neu}^{\gamma+2}(p,\theta,\tau)$ be a solution to \eqref{maineq} with $u_0=0, f=0$, and $g=0$. By \eqref{2511132025}, we get $D_du=0$, which implies that $u$ satisfies
\begin{equation*}
    du= \sum_{i,j<d} a^{ij}D_{ij}u dt + \sum_{i<d} \sigma^{ik} D_i u dw_t^k.
\end{equation*}
Let $m\in\bN$ such that $2m+\gamma\geq2$. We can choose $c=c(d,p,\theta,\gamma,m)>0$ so that the $m$-th power $\cL^m$ of the operator $\cL$ considered in the existence part is an isomorphism on the given space.
Since $\cL$ commutes with $D_{i}$ when $i<d$, $\cL^{-m}u$ satisfies
\begin{equation} \label{eq9071428}
    d(\cL^{-m}u)= \sum_{i,j<d} a^{ij}D_{ij}(\cL^{-m}u) dt + \sum_{i<d} \sigma^{ik} D_i (\cL^{-m}u) dw_t^k
\end{equation}
with $\cL^{-m}u(0,\cdot)=0$.
As in the proof of Lemma \ref{lem4071423}, if necessary, we may first mollify in the $x_d$-direction. Then for each fixed $x_d$, the mollified function can be viewed as satisfying \eqref{eq9071428} on $\bR^{d-1}$.
Since $2m+\gamma\geq 2$ and $\cL^{-m}u\in \frH_{neu}^{2m+\gamma}(p,\theta,\tau)$, we can apply the uniqueness result \cite[Theorem 4.10]{Kry99} to obtain that $\cL^{-m}u=0$, which leads to $u=0$.
The lemma is proved.
\end{proof}

We next turn to the initial value problem. 

\begin{lemma} \label{lem5141725}
    Let $\gamma\in\bR$, $p\in[2,\infty)$, $T\in(0,\infty)$, $\theta\in(d-1,d-1+p)$, and $\tau\leq T$ be a bounded stopping time. Then for any $u_0\in U_{p,\theta}^{\gamma+2}$, there is a unique solution $u\in \frH_{neu}^{\gamma+2}(p,\theta,\tau)$ to
    \begin{equation} \label{eq9060105}
        \begin{cases}
du= \Delta u  dt, \quad t>0, x\in \bR^d_+,
\\
D_{d}u|_{x_d=0}=0,
\\
u(0,x)=u_0.
\end{cases}
    \end{equation}
    Moreover, we have
    \begin{align} \label{eq5141723}
    \|u\|_{\frH_{neu}^{\gamma+2}(p,\theta,\tau)} &\leq N\|u_0\|_{U_{p,\theta}^{\gamma+2}},
    \end{align}
where $N=N(d,\gamma,p,\theta,T)$.
\end{lemma}

\begin{proof}
The uniqueness follows directly from Lemmas \ref{lem5141719} and \ref{lem4201603}. Thus, we only need to prove the existence of a solution.

    By the extension theorem (see e.g. \cite[Theorem 1.6]{CLSW26}), there is $v\in \frH_{neu}^{\gamma+2}(p,\theta,\tau)$ such that $dv=F dt$ with $v(0,\cdot)=u_0$, where $F\in \bH_{p,\theta}^{\gamma}(\tau)$. Moreover, we have
    \begin{equation} \label{eq5141721}
         \|v\|_{\frH_{neu}^{\gamma+2}(p,\theta,\tau)} \leq N\|u_0\|_{U_{p,\theta}^{\gamma+2}}.
    \end{equation}
    We remark that although the result in \cite[Theorem 1.6]{CLSW26} is proved for deterministic spaces, the same proof applies to the stochastic setting as well. By Lemmas \ref{lem5141719} and \ref{lem4201603}, there is $w\in \frH_{neu}^{\gamma+2}(p,\theta,\tau)$ satisfying
    \begin{equation*}
        dw=(\Delta w + \Delta v-F)dt
    \end{equation*}
    with $w(0,\cdot)=0$. Moreover, by \eqref{eq5141721},
    \begin{align*}
        \|w\|_{\frH_{neu}^{\gamma+2}(p,\theta,\tau)} \leq N \| \Delta v-F\|_{\bH_{p,\theta}^{\gamma}(\tau)} \leq N \|v\|_{\frH_{neu}^{\gamma+2}(p,\theta,\tau)} \leq N\|u_0\|_{U_{p,\theta}^{\gamma+2}}.
    \end{align*}
    Notice that $u:=v+w$ is a solution to \eqref{eq9060105} satisfying \eqref{eq5141723}.
    The lemma is proved.
\end{proof}

\begin{proof}[Proof of Theorem \ref{thm_main}]

\textbf{1.} Linear case: $f$ and $g$ are independent of $u$.

This case can be deduced by combining Lemmas \ref{lem4162045}-\ref{lem5141725}.

\textbf{2.} Semi-linear case when $\tau=T$ is a constant.

We first prove the existence result. By the linear case, we can define 
$u^0 \in \frH_{neu}^{\gamma+2}(p,\theta,T)$ as the solution to \eqref{maineq} with the inhomogeneous terms $f(0)$ and $g(0)$, which are independent of $u$.
Similarly, for $n\geq0$, we define $u^{n+1}$ inductively as the solution to
\begin{equation*}
du^{n+1}=\left(a^{ij}(t)D_{ij}u^{n+1}+f(u^n)\right)dt + \left(\sigma^{ik}(t) D_iu^{n+1} + g^k(u^n) \right) dw_t^k 
\end{equation*}
with the initial condition $u^{n+1}(0,\cdot)=u_0$.
Thus, $(u^{n+1}-u^n)(0,\cdot)=0$ and
\begin{align*}
d(u^{n+1}-u^n)&=\left(a^{ij}(t)D_{ij}(u^{n+1}-u^n)+ f(u^{n})-f(u^{n-1})\right)dt 
\\
&\quad+ \sum_{k=1}^\infty\left(\sigma^{ik}(t)D_i(u^{n+1}-u^{n}) + g^k(u^n)-g^k(u^{n-1})\right)dw_t^k.
\end{align*}
Here, we denote $u^{-1}:=0$.
Then by the linear case and Assumption \ref{as_semi}, for each $t\leq T$,
\begin{align} \label{eq4171440}
&\|u^{n+1} - u^n\|^p_{\frH_{neu}^{\gamma+2}(p,\theta,t)} \nonumber
\\
&\leq N\left( \|f(u^n)-f(u^{n-1})\|^p_{\bH_{p,\theta}^{\gamma}(t)}+ \| g(u^n)-g(u^{n-1})\|^p_{{\bH}_{p,\theta}^{\gamma+1,(1)}(t,l_2)} \right) \nonumber
\\
&\leq N \varepsilon \left(\|x_d^{-1}D_d(u^n-u^{n-1})\|^p_{\bH_{p,\theta}^{\gamma+1}(t)} + \|D_x^2(u^n-u^{n-1})\|^p_{\bH_{p,\theta}^{\gamma}(t)} \right) \nonumber
\\
&\quad+ N_\varepsilon \|u^n-u^{n-1}\|^p_{\bH_{p,\theta}^{\gamma}(t)} \nonumber
\\
&\leq N\varepsilon \|u^n-u^{n-1}\|^p_{\frH_{neu}^{\gamma+2}(p,\theta,t)} + N_\varepsilon \left[\int_0^t \|u^n-u^{n-1}\|_{\frH_{neu}^{\gamma+2}(p,\theta,s)}^p ds\right],
\end{align}
where $N_\varepsilon$ depends also on $\varepsilon$, but it is independent of $n$.
For the last inequality above, we used
\begin{align*}
\|u\|_{\bH_{p,\theta}^{\gamma}(t)}^p \leq \int_0^t \left( \bE \sup_{r\leq s} \|u(r,\cdot)\|_{H_{p,\theta}^{\gamma}}^p \right) ds \leq N \int_0^t \|u\|_{\frH_{neu}^{\gamma+2}(p,\theta,s)}^p ds,
\end{align*}
which follows from \eqref{eq01072253}.
By iterating this procedure in $n$, for any $t\leq T$,
\begin{align*}
&\|u^{n+1}-u^n\|^p_{\frH_{neu}^{\gamma+2}(p,\theta,t)} \leq (N\varepsilon)^n \|u^1-u^0\|^p_{\frH_{neu}^{\gamma+2}(p,\theta,t)} 
\\
&\quad+ \sum_{k=1}^n \binom{n}{k} (N\varepsilon)^{n-k} N_\varepsilon^k \left[\int_0^t \frac{(t-s)^{k-1}}{(k-1)!} \|u^1-u^0\|_{\frH_{neu}^{\gamma+2}(p,\theta,s)}^p ds\right],
\end{align*}
where $\binom{n}{k}$ denotes the standard binomial coefficient.
Hence, we get
\begin{align*}
\|u^{n+1}-u^{n}\|^p_{\frH_{neu}^{\gamma+2}(p,\theta,T)}  &\leq \left((2N\varepsilon)^n \sup_{k\geq0} \frac{(N_\varepsilon T/\varepsilon)^{k}}{k!} \right) \|u^1-u^0\|_{\frH_{neu}^{\gamma+2}(p,\theta,T)}^p.
\end{align*}
By taking a sufficiently small $\varepsilon>0$ such that $2N\varepsilon<1/2$, since the supremum above is finite,
\begin{equation*}
    \sum_{n=1}^\infty \|u^{n+1}-u^{n}\|_{\frH_{neu}^{\gamma+2}(p,\theta,T)} < \infty.
\end{equation*}
Thus, we conclude that $u^n$ is a Cauchy sequence in $\frH_{neu}^{\gamma+2}(p,\theta,T)$, and its limit, say $u$, is a solution to \eqref{maineq}.

Next, we prove the uniqueness of solution.
 Suppose that $u$ and $v$ are solutions to \eqref{maineq}. Then $w:=u-v$ satisfies
\begin{equation*}
dw=\left( a^{ij}(t)D_{ij}w + f(u)-f(v) \right) dt + \left( \sigma^{ik}(t)D_iw + g^k(u) - g^k(v) \right) dw_t^k
\end{equation*}
with the zero initial condition.
Similar to \eqref{eq4171440}, one can obtain that for $\varepsilon>0$ and $t\leq T$,
\begin{align*}
    \|w\|^p_{\frH_{neu}^{\gamma+2}(p,\theta,t)} \leq N\varepsilon \|w\|^p_{\frH_{neu}^{\gamma+2}(p,\theta,t)} + N_\varepsilon \int_0^t \|w\|_{\frH_{neu}^{\gamma+2}(p,\theta,s)}^p ds.
\end{align*}
 By taking sufficiently small $\varepsilon>0$ and applying Gronwall's lemma, we get $w=0$, which proves the uniqueness result.  

Lastly, we deal with estimate \eqref{main_est}. 
Let $u\in \frH_{neu}^{\gamma+2}(p,\theta,T)$ be the solution to \eqref{maineq}. We also consider $u^0\in \frH_{neu}^{\gamma+2}(p,\theta,T)$, which is a solution to \eqref{maineq} with the inhomogeneous terms $f(0)$ and $g(0)$. As in \eqref{eq4171440}, for any $t\leq T$,
\begin{align*}
    \|u-u^{0}\|^p_{\frH_{neu}^{\gamma+2}(p,\theta,t)} &\leq N \left( \|f(u)-f(0)\|_{\bH_{p,\theta}^{\gamma}(t)}^p + \|g(u)-g(0)\|_{{\bH}_{p,\theta}^{\gamma+1,(1)}(t,l_2)}^p \right) \nonumber
\\
&\leq N\varepsilon \|u\|^p_{\frH_{neu}^{\gamma+2}(p,\theta,t)} + N_\varepsilon \int_0^t \|u\|_{\frH_{neu}^{\gamma+2}(p,\theta,s)}^p ds.
\end{align*}
Due to the linear case,
\begin{align*}
    \|u^{0}\|^p_{\frH_{neu}^{\gamma+2}(p,\theta,t)} \leq N\left(\|u_0\|_{U_{p,\theta}^{\gamma+2}}^p + \|f(0)\|_{\bH_{p,\theta}^\gamma(T)}^p + \|g(0)\|_{\bH_{p,\theta}^{\gamma+1,(1)}(T,l_2)}^p \right),
\end{align*}
which yields that
\begin{align*}
\|u\|^p_{\frH_{neu}^{\gamma+2}(p,\theta,t)} 
&\leq  N \int_0^t \|u\|^p_{\frH_{neu}^{\gamma+2}(p,\theta,s)} ds
\\
&\quad+ N \|u_0\|_{U_{p,\theta}^{\gamma+2}}^p + N \|f(0)\|_{\bH_{p,\theta}^\gamma(T)}^p + N \|g(0)\|_{\bH_{p,\theta}^{\gamma+1,(1)}(T,l_2)}^p,
\end{align*}
for any $t\leq T$.
Again by Gronwall's lemma, we obtain the desired estimate.

\textbf{3.} Semi-linear case for general $\tau\leq T$. 

We first consider the existence result and estimate \eqref{main_est}. 
Let
\begin{equation*}
    \bar{f}(t,x,u):=1_{t\leq \tau}f(t,x,u), \quad \bar{g}(t,x,u):=1_{t\leq \tau}g(t,x,u).
\end{equation*}
Then one can easily check that $\bar{f}$ and $\bar{g}$ still satisfy Assumption \ref{as_semi} with $\tau=T$. Thus, by the case $\tau=T$, there exists a solution $u\in \frH_{neu}^{\gamma+2}(p,\theta,T)$ to equation \eqref{maineq} with $\bar{f}$, $\bar{g}$, and $T$, in place of $f$, $g$, and $\tau$, respectively. Since $\bar{f}(u)=f(u)$ and $\bar{g}(u)=g(u)$ for $t\leq \tau$, we see that $u\in \frH_{neu}^{\gamma+2}(p,\theta,T)$ is a solution to equation \eqref{maineq}.
Thus, the restriction of $u$ to $[0,\tau]$ belongs to $\frH_{neu}^{\gamma+2}(p,\theta,\tau)$ and satisfies \eqref{maineq} on $(0,\tau]$.
Moreover, the estimate easily follows from the relation $\|u\|_{ \frH_{neu}^{\gamma+2}(p,\theta,\tau)} \leq  \|u\|_{\frH_{neu}^{\gamma+2}(p,\theta,T)}$.

Next we prove the uniqueness result. 
Suppose that $u\in \frH_{neu}^{\gamma+2}(p,\theta,\tau)$ is a solution to \eqref{maineq}.
 By the case when $\tau=T$, there is a solution $\tilde{u}\in \frH_{neu}^{\gamma+2}(p,\theta,T)$ to
 \begin{equation} \label{eq4171718}
     d\tilde{u}=\left(a^{ij}(t)D_{ij}\tilde{u}+\bar{f}(u)\right)dt + \left( \sigma^{ik}(t)D_i\tilde{u} + \bar{g}^k(u)\right) dw_t^k
 \end{equation}
 with the initial condition $\tilde{u}(0,\cdot)=u_0$.
Here, we remark that we introduced $\bar{f}(u)$ and $\bar{g}(u)$, not $\bar{f}(\tilde{u})$ and $\bar{g}(\tilde{u})$. Since $\tilde{u}-u$ satisfies
\begin{equation*}
    d(\tilde{u}-u)=\left(a^{ij}(t)D_{ij}(\tilde{u}-u) \right)dt + \left( \sigma^{ik}(t)D_i(\tilde{u}-u) \right) dw_t^k
\end{equation*}
with the zero initial condition in $(0,\tau]$, the uniqueness result of the linear case leads to $u=\tilde{u}$ in $(0,\tau]$. Therefore, in equation \eqref{eq4171718}, we can replace $\bar{f}(u)$ and $\bar{g}(u)$ by $\bar{f}(\tilde{u})$ and $\bar{g}(\tilde{u})$, respectively. That is, $\tilde{u}$ is a solution to the semi-linear equation \eqref{maineq} in $\frH_{neu}^{\gamma+2}(p,\theta,T)$. Therefore, the uniqueness result for general $\tau$ follows from the case $\tau=T$.
The proof of the theorem is completed.
\end{proof}

\subsection{Proof of Sobolev-H\"older embedding results} \label{sec_embed}

In this section, we aim to prove Theorems \ref{thm_bdry} and \ref{thm_int}.

\begin{proof}[Proof of Theorem \ref{thm_bdry}]
We may assume that $\gamma\leq0$.
    Let $du = Fdt + G^k dw_t^k$,
    where $F\in \bH_{p,\theta}^{\gamma}(\tau)$, $G\in \bH_{p,\theta}^{\gamma+1,(1)}(\tau,l_2)$. For a distribution $v\in \cD'(\bR^d_+)$, we write $v_R:=I_Rv$. Then obviously,
    \begin{equation*}
        du_R=F_R dt + G_R^kdw_t^k,
    \end{equation*}
    with $u_R(0,\cdot) = u_{0R}$. Note that by \eqref{defH}, the operator $I_{2R}$ is bounded from $H_{p,\theta}^{\gamma',(n)}$ to $H_{p,\theta_0}^{\gamma',(n)}$ if $\gamma'\leq n$. This, \eqref{eq5121127}, and the relation $I_R v=I_R I_{2R} v$ lead to
    \begin{eqnarray*}
        u_R \in F_{p,2}^{\gamma+2}(\bR^d_+,x_d^{\theta_0-d}), \quad F_R \in F_{p,2}^{\gamma}(\bR^d_+,x_d^{\theta_0-d}), \quad G_R \in F_{p,2}^{\gamma+1}(\bR^d_+,x_d^{\theta_0-d}, l_2),
    \end{eqnarray*}
    for each $t\leq \tau$, where $F_{p,2}^{\gamma+1}(\bR^d_+,x_d^{\theta_0-d}, l_2)$ is the natural extension to the $l_2$-valued function space. For the initial data $u_{0R}$, using \eqref{eq5121127}, the real interpolation of operators (see e.g. \cite[Theorem C.3.3]{HNVW16}), and
    \begin{align*}
        (\widetilde{H}_{p,\theta}^{\gamma+2,(2)},H_{p,\theta}^{\gamma})_{1/p,p} \subset ({H}_{p,\theta}^{\gamma+2,(2)},H_{p,\theta}^{\gamma})_{1/p,p},
    \end{align*}
    one can deduce that for each $\omega\in \Omega$,
    \begin{align*}
        u_{0R} \in (F_{p,2}^{\gamma+2}(\bR^d_+,x_d^{\theta_0-d}),F_{p,2}^{\gamma}(\bR^d_+,x_d^{\theta_0-d}))_{1/p,p}.
    \end{align*}
By \cite[Proposition 5.6]{LMV18}, there is an extension operator $\cE$ from $F_{p,2}^{\gamma'}(\bR^d_+,x_d^{\theta_0-d})$ to $F_{p,2}^{\gamma'}(\bR^d,|x_d|^{\theta_0-d})$. We note that if $\gamma'$ belongs to a certain range, for example, $\gamma'\in[\gamma,\gamma+2]$, $\cE$ can be chosen as a universal operator independent of $\gamma'$. In particular, due to \cite[(5.1)]{LMV18},
\begin{equation*}
    d(\cE u_R)=(\cE F_R) dt + (\cE G_R)^kdw_t^k, \quad \cE u_R(0,\cdot) = \cE u_{0R},
\end{equation*}
where for $t\leq \tau$,
        \begin{eqnarray*}
        &\cE u_R \in F_{p,2}^{\gamma+2}(\bR^d,|x_d|^{\theta_0-d}), \quad \cE F_R \in F_{p,2}^{\gamma}(\bR^d,|x_d|^{\theta_0-d}),&
        \\
        &\cE G_R \in F_{p,2}^{\gamma+1}(\bR^d,|x_d|^{\theta_0-d}, l_2).&
    \end{eqnarray*}
    Moreover, by the real interpolation of operators and \cite[Proposition 6.1]{MV12},
    \begin{align*}
        \cE u_{0R} \in (F_{p,2}^{\gamma+2}(\bR^d,|x_d|^{\theta_0-d}),F_{p,2}^{\gamma}(\bR^d,|x_d|^{\theta_0-d}))_{1/p,p} = B_{p,p}^{\gamma+2-2/p}(\bR^d,|x_d|^{\theta_0-d}).
    \end{align*}
    By \cite[Theorem 2.3]{R25}, when $\theta_0>d$, the weighted spaces are embedded into the corresponding unweighted spaces:
\begin{eqnarray*}
    &\cE u_R \in L_p((0,\tau],\cP;H_{p}^{\gamma+2-\frac{\theta_0-d}{p}}(\bR^d)), \quad \cE u_{0R} \in L_p(\Omega,\cF_0;B_{p,p}^{\gamma+2-\frac{2}{p}-\frac{\theta_0-d}{p}}(\bR^d)),&
    \\
    &\cE F_R \in L_p((0,\tau],\cP;H_{p}^{\gamma-\frac{\theta_0-d}{p}}(\bR^d)), \quad \cE G_R \in L_p((0,\tau],\cP;H_{p}^{\gamma+1-\frac{\theta_0-d}{p}}(\bR^d,l_2)).&
\end{eqnarray*}
When $\theta_0=d$, these relations follow immediately.
Thus, by \cite[Theorem 7.2]{Kry99},
\begin{align*}
    \bE \|\cE u_R\|_{C^{\alpha-1/p}\left([0,\tau];H_p^{\gamma+2-\frac{\theta_0-d}{p}-2\beta}(\bR^d)\right)}^p <\infty.
\end{align*}
We note that the conclusion of \cite{Kry99} remains valid when the initial data space is replaced by the corresponding optimal Besov trace space.
It remains to apply the standard Sobolev-H\"older embedding theorem in space.
The proof of the first part, \eqref{eq5251945}, is completed.

Now we consider the global estimate \eqref{eq5271711}.
Let us denote
\begin{equation*}
    du=Fdt+G^kdw_t^k,
\end{equation*}
where $F\in \bL_{p,\theta}(\tau)$ and $G\in {\bH}_{p,\theta}^{1,(1)}(\tau,l_2)$. Unlike in the proof of \eqref{eq5251945}, using Lemma \ref{lem5141224}, we obtain the following directly without using the operator $I_R$:
        \begin{eqnarray*}
        &\cE u \in F_{p,2}^{2}(\bR^d,|x_d|^{\theta-d}), \quad \cE F \in F_{p,2}^{0}(\bR^d,|x_d|^{\theta-d}),&
        \\
        &\cE G \in F_{p,2}^{1}(\bR^d,|x_d|^{\theta-d}, l_2), \quad \cE u_{0} \in B_{p,p}^{2-2/p}(\bR^d,|x_d|^{\theta-d}).&
    \end{eqnarray*}
Thus, one just needs to repeat the above argument with $\gamma=0$. The theorem is proved.
\end{proof}

Lastly, we deal with the interior regularity result.

\begin{proof}[Proof of Theorem \ref{thm_int}]
Let $du = Fdt + G^k dw_t^k$, 
    where $F\in \bH_{p,\theta}^{\gamma}(\tau)$, $G\in \bH_{p,\theta}^{\gamma+1,(1)}(\tau,l_2)$. For $\chi\in C_c^\infty(\bR^d_+)$, we have
    \begin{equation*}
        d(\chi u) = \chi Fdt + \chi G^k dw_t^k, \quad \chi u(0,\cdot)= \chi u_0,
    \end{equation*}
    and by \cite[Remark 1.12]{Khalf}, one can find that
    \begin{eqnarray*}
        &\chi u\in \bH_{p}^{\gamma+2}(\tau), \quad \chi F \in \bH_{p}^{\gamma}(\tau), \quad \chi G \in \bH_{p}^{\gamma+1}(\tau,l_2),
    \end{eqnarray*}
    where $\bH_{p}^{\gamma}(\tau) := L_p((0,\tau], \cP; H_p^{\gamma})$ and $\bH_{p}^{\gamma}(\tau,l_2) := L_p((0,\tau], \cP; H_p^{\gamma}(l_2))$. In other words, the functions can be viewed as elements of $\cD'(\bR^d)$. For the initial data, using
    \begin{equation*}
         (\widetilde{\fH}_{p,\theta}^{\gamma+2,(2)},\fH_{p,\theta}^{\gamma})_{1/p,p} \subset ({\fH}_{p,\theta}^{\gamma+2,(2)},\fH_{p,\theta}^{\gamma})_{1/p,p},
    \end{equation*}
    we have
    \begin{align*}
        \chi u_0 \in ({\fH}_{p}^{\gamma+2},\fH_{p}^{\gamma})_{1/p,p} = \fB_{p,p}^{\gamma+2-2/p},
    \end{align*}
    where $\fH_{p}^{\gamma}:= L_p(\Omega,\cF_0;H_{p}^{\gamma})$ and $\fB_{p,p}^{\gamma}:= L_p(\Omega,\cF_0;B_{p,p}^{\gamma})$. Thus, one can apply \cite[Theorem 7.2]{Kry99} to obtain that
    \begin{equation*}
        \bE \|\chi u\|_{C^{\alpha-1/p}\left([0,\tau];H_p^{\gamma+2-2\beta}(\bR^d)\right)}^p <\infty.
    \end{equation*}
    Now \eqref{eq5251946} follows from the standard Sobolev-H\"older embedding theorem in space. The theorem is proved.
\end{proof}

\appendix

\section*{Acknowledgement} \pdfbookmark{Acknowledgement}{Acknowledgement}
 
The author is grateful to Sungbae Jeong for finding several typos in the earlier version of the article.




\end{document}